\documentclass[10pt,reqno]{amsart}
\usepackage{amssymb}
\usepackage{amsmath}
\usepackage{amsthm}
\usepackage{mathtools}
\usepackage{graphicx}

\usepackage{cite} 

\usepackage{subfigure}

\usepackage{caption}

\usepackage{bm}

\usepackage{enumitem}

\usepackage[colorlinks=true,urlcolor=blue,
citecolor=red,linkcolor=blue,linktocpage,pdfpagelabels,
bookmarksnumbered,bookmarksopen]{hyperref}

\usepackage{ifthen} 

\provideboolean{shownotes} 
\setboolean{shownotes}{true} 
\newcommand{\margnote}[1]{
\ifthenelse{\boolean{shownotes}}%
{\marginpar{\raggedright\tiny\texttt{#1}}}%
{}%
}

\newcommand{\hole}[1]{
\ifthenelse{\boolean{shownotes}}%
{\begin{center} \fbox{ \rule {.25cm}{0cm}
\rule[-.1cm]{0cm}{.4cm} \parbox{.85\textwidth}{\begin{center}
\texttt{#1}\end{center}} \rule {.25cm}{0cm}}\end{center}}
{}
}

\usepackage[titletoc,title]{appendix}

\numberwithin{equation}{section} 

\newcounter{cont}[section]

\newtheorem{theorem}[cont]{Theorem}
\newtheorem{proposition}[cont]{Proposition}

\newtheorem{corollary}[cont]{Corollary}
\newtheorem{lemma}[cont]{Lemma}
\newtheorem*{mtheorem}{Theorem}

\theoremstyle{definition}
\newtheorem{remark}[cont]{Remark}
\newtheorem{definition}[cont]{Definition}

 \theoremstyle{remark}

\renewcommand{\Re}{\mathrm{Re}\,} 
\renewcommand{\Im}{\mathrm{Im}\,}

\newcommand{\e}{\varepsilon}
\newcommand{\ep}{\epsilon}
\newcommand{\vep}{\varepsilon}

\newcommand{\N}{\mathbb{N}}
\newcommand{\R}{\mathbb{R}}
\newcommand{\C}{\mathbb{C}}

\newcommand{\cL}{{\mathcal{L}}}

\newcommand{\cM}{{\mathcal{M}}}
\newcommand{\cD}{{\mathcal{D}}}
\newcommand{\cA}{{\mathcal{A}}}

\newcommand{\ess}{\sigma_\mathrm{\tiny{ess}}}
\newcommand{\ptsp}{\sigma_\mathrm{\tiny{pt}}}

\begin{document}

\title[Stability of weak dispersive shock profiles for QHD]{Spectral stability of weak dispersive shock profiles for quantum hydrodynamics with nonlinear viscosity}

\author[R. Folino]{Raffaele Folino}

\address[R. Folino]{Departamento de Matem\'aticas y Mec\'anica\\Instituto de 
Investigaciones en Matem\'aticas Aplicadas y en Sistemas\\Universidad Nacional Aut\'onoma de 
M\'exico\\Circuito Escolar s/n, Ciudad Universitaria C.P. 04510 Cd. Mx. (Mexico)}

\email{folino@mym.iimas.unam.mx}

\author[R. G. Plaza]{Ram\'on G. Plaza}

\address[R. G. Plaza]{Departamento de Matem\'aticas y Mec\'anica\\Instituto de 
Investigaciones en Matem\'aticas Aplicadas y en Sistemas\\Universidad Nacional Aut\'onoma de 
M\'exico\\Circuito Escolar s/n, Ciudad Universitaria C.P. 04510 Cd. Mx. (Mexico)}

\email{plaza@mym.iimas.unam.mx}

\author[D. Zhelyazov]{Delyan Zhelyazov}

\address[D. Zhelyazov]{Departamento de Matem\'aticas y Mec\'anica\\Instituto de 
Investigaciones en Matem\'aticas Aplicadas y en Sistemas\\Universidad Nacional Aut\'onoma de 
M\'exico\\Circuito Escolar s/n, Ciudad Universitaria C.P. 04510 Cd. Mx. (Mexico)}

\email{delyan.zhelyazov@iimas.unam.mx}




\begin{abstract} 
This paper studies the stability of weak dispersive shock profiles for a quantum hydrodynamics system in one space dimension with nonlinear viscosity and dispersive (quantum) effects due to a Bohm potential. It is shown that, if the shock amplitude is sufficiently small, then the profiles are spectrally stable. This analytical result is consistent with numerical estimations of the location of the spectrum \cite{LaZ21b}. The proof is based on energy estimates at the spectral level, on the choice of an appropriate weighted energy function for the perturbations involving both the dispersive potential and the nonlinear viscosity, and on the montonicity of the dipersive profiles in the small-amplitude regime.
\end{abstract}


\maketitle

\section{Introduction}\label{sec:intro}

Consider the following one-dimensional quantum hydrodynamics (QHD) system with nonlinear viscosity in Eulerian coordinates:
\begin{equation}
\label{QHD-E}
\begin{aligned}
\rho_t + m_x &= 0,   \\
m_t + \Big( \frac{m^2}{\rho} + p(\rho) \Big)_x &= \ep \mu \rho \Big( \frac{m_x}{\rho}\Big)_x + \ep^2 k^2 \rho \left( \frac{(\sqrt{\rho})_{xx}}{\sqrt{\rho}}\right)_x,\\       
\end{aligned} 
\end{equation}
with $x \in \R$ and $t>0 $ denoting space and time, respectively, and where the unknown scalar functions $\rho= \rho(x,t) > 0$ 
and $m= m(x,t) = \rho(x,t) u(x,t)$ denote the density and momentum fields, respectively, where $u = u(x,t)$ is the velocity. 
Here, $p=p(\rho)$ denotes the pressure function, also known as equation of state. 
It will be assumed that $p(\rho)=\rho^{\gamma}$, with constant $\gamma \geq 1$. The constants $0 < \ep \ll1$, $\mu > 0$ and $k > 0$ determine the viscosity, 
$\ep \mu$, and dispersion (or capillarity), $\ep^2 k^2$, coefficients, respectively. 
The dispersive third order term corresponds to a quantum Bohm potential \cite{Boh52a,Boh52b}, 
whereas the nonlinear viscosity term is motivated by the theory of superfluidity (see Khalatnikov \cite{Khlt89}). 

Systems of the form \eqref{QHD-E} appear in many areas of physics to describe a compressible, viscous fluid where quantum effects appear at a macroscopic scale. This is the case, for example, in phenomena like Bose-Einstein condensation \cite{DGPS99,GrantJ73}, the theory of superfluidity \cite{Khlt89,Land41,LanLif6}, or in the (hydrodynamical) modeling of semiconductor devices at nanoscales \cite{GarC94,FrZh93}. The form of the dispersive term appearing in \eqref{QHD-E} comes from considering a normalized quantum Bohm potential, $\tfrac{1}{2}\ep^2 (\sqrt{\rho})_{xx}/\sqrt{\rho}$ (cf. \cite{Boh52a,Boh52b}), and represents a quantum correction to the classical pressure. It provides the model with a nonlinear third order dispersive term. The nonlinear viscosity term chosen here is motivated by the theory of superfluidity (cf. \cite{Khlt89}, p. 109; see also the discussion in \cite{JuMi07}) and describes the interactions between superfluid flow (without any loss of kinetic energy) and normal flow. It can also be interpreted as the interaction of the fluid with a background (see Landau and Lifshitz \cite{LanLif6} and Khalatnikov \cite{Khlt89} for more information). The constants $\mu$ and $k$ are assumed to be of order $O(1)$ with respect to the physical constant $\ep > 0$. The latter is proportional to the reduced Planck constant, $\hbar > 0$, and it is assumed to be fixed and small, $0 < \ep \ll 1$.

In the context of QHD models, an important field of study is related to the emergence of shock waves dominated by dispersion rather than dissipation. They are known as \emph{dispersive shocks}. For instance, purely dispersive shocks (without viscosity) were first studied in \cite{GuPi74,Sgdv64} (see also \cite{Gas01,HACCES06,HoAb07} for later developments). Motivated by classical fluid theory, a new approach focuses on the interaction between viscosity and dispersion (see, e.g., \cite{DiMu17,GMOB22,Zhel-preprint}), where viscous-dispersive shocks play a fundamental role. In recent works, the third author and collaborators have studied the existence and stability properties of viscous-dispersive shock profiles for QHD systems, first for the case of \emph{linear} viscosity (see \cite{LMZ20a,LMZ20b}) and, more recently, for the system \eqref{QHD-E} with nonlinear viscosity under consideration here (see \cite{LaZ21b,LaZ21a}). For instance, system \eqref{QHD-E} can be recast in conservation form in terms of the conserved quantities $\rho$ and $u$ (see \cite{LaZ21a,LaZ21b} or Section \ref{secconserv} below). Dispersive shock profiles are traveling wave solutions to \eqref{QHD-E} of the form 
\[
\rho(x,t)=R \Big( \frac{x-st}{\ep} \Big),\qquad \qquad u(t,x)=U\Big( \frac{x-st}{\ep} \Big),
\]
where $s \in \R$ is the shock speed, and for which there exist asymptotic limits
\[
	R^{\pm}=\lim_{\xi \rightarrow \pm \infty}R(\xi), \qquad \qquad  U^{\pm}=\lim_{\xi \rightarrow \pm \infty}U(\xi).
\]
The triplet $(R^\pm, U^\pm,s)$ defines a classical (Lax) shock of the underlying hyperbolic system in the absence of viscosity and dispersion (see Section \ref{secshocks} below). The authors of \cite{LaZ21b,LaZ21a} pay attention to the interplay of the dispersive and the viscosity terms and to its effects on the existence and the structural properties of these shock profiles. 

One of the most important properties of these traveling fronts is their \emph{stability} as solutions to the viscous-dispersive system \eqref{QHD-E} under small perturbations. Before defining what we understand as spectral stability, it is worth mentioning that the results and the methodology for viscous dispersive shocks are strongly influenced by the stability theory of purely viscous shock profiles, which we now briefly review. The literature on the stability of viscous shocks is vast and dates back to the work by Il’in and Ole\u{\i}nik \cite{IO64} in the scalar case (see also Sattinger \cite{Sat76}). A fundamental advance in the theory was the use of energy methods to study general systems with parabolic viscosity by Goodman \cite{Go86}, and independently by Matstumura and Nishihara \cite{MN85} for the one-dimensional compress­ible viscous gas system. Their results hold for weak shock profiles and under the assumption of zero mass perturba­tions. Liu \cite{L85} later extended Goodman's results to non-zero mass perturbations via a characteristic-energy method, and introduced the diffusion waves transporting the mass. Szepessy and Xin \cite{SX93} provided the first complete stability result for viscous shocks by relaxing the stringent conditions on the initial perturbations required in \cite{Go86,MN85}. Around the mid-1990’s, Liu \cite{L97} developed the technique of using pointwise bounds on the Green's function of the linearized operator around the wave and proved nonlinear stability for genuinely non-linear weak Lax shocks with artificial viscosity.

The seminal paper of Alexander, Gardner and Jones \cite{AGJ90} contained the basis of a dynamical systems approach (known as \emph{Evans function meth­ods}) to the stability of trav­eling waves; see also \cite{San02,KaPro13} for comprehensive reviews. However, the accumulation of essential spectrum around the eigenvalue zero prevented a direct application of Evans function techniques to viscous shocks. This difficulty was overcome by Gard­ner and Zumbrun \cite{GZ98} who proved a technical result known as the \emph{Gap lemma} (see also \cite{KS98}), which allows to define the Evans function near the essential spectrum. Since then, a theory developed mainly by K. Zumbrun and collaborators (see, e.g. \cite{H99,ZH98,MZ04a,MZ04b,MZ03,Z01,Z04,Z07} and the many references therein) has combined pointwise Green's function bounds with Evans function tools to reduce the nonlinear stability problem of shock profiles to the verification of the strong spectral stability for the linearized operator around the wave. In other words, the theory, which applies to several general classes of shocks, shows that spectral stability of the linearized operator implies nonlinear stability. As far as we know, the (now fundamental) spectral stability condition has been verified only for sufficiently weak general viscous and relaxation shock profiles (see the references \cite{HuZ02,MZ09} which apply the energy method, as well as \cite{PZ04,FSz02,FSz10} for different approaches based on dynamical systems tools), or for large-amplitude shocks in very specific systems, either analytically, using \emph{ad hoc} energy estimates particular to the system (cf. \cite{MN85}), or by numerical estimations (see \cite{HLyZ09,HLZ10}). In sum, the spectral stability condition has established itself as a fundamental property for viscous shock profiles; its verification is far from trivial and, in particular, in the case of large-amplitude shocks, it has been quite elusive from a theoretical viewpoint.

In the case of shock profiles for systems exhibiting viscosity and capillarity, the theory is much less developed. The only known stability results pertain to scalar equations \cite{PaW04,HZ00}, isentropic fluid dynamics with constant capillarity \cite{Hu09}, or with constant capillarity and constant viscosity \cite{KhdPhD89,ZLY16,Z00}. The analysis by Humpherys \cite{Hu09} applies spectral energy methods (those are, energy estimates at the level of the spectral equations) to prove the spectral stability of sufficiently small amplitude shocks. In this work, the proof is based on the monotonicity of the profiles and on finding a particular energy for this system. The analysis by Zumbrun \cite{Z00} warrants note as the only result, up to our knowledge, for \emph{large amplitude} shocks exhibiting viscosity and capillarity, again obtained by energy estimates special to the model, which establishes the stability of stationary phase-transitional large amplitude shocks of an isentropic viscous-capillary van der Waals model introduced by Slemrod \cite{Sl84a}. From a theoretical perspective, the corresponding pointwise Green's function bound method, which establishes nonlinear stability from assuming the spectral stability condition, has been developed only for \emph{scalar} viscous-dispersive shocks (see Howard and Zumbrun \cite{HZ00}). As far as we know, the systems case is still open. 

In the particular case of QHD models, the third author and collaborators have studied not only the existence but also the stability properties of viscous-dispersive shock profile solutions to system \eqref{QHD-E} with nonlinear viscosity (see \cite{LaZ21b,LaZ21a}) and to a related model with a linear (non-physical) viscosity term as well (cf. \cite{LMZ20a,LMZ20b}). The authors consider both monotone (small-amplitude) and oscillatory (larger amplitude) shock profiles. In these works, the authors provide numerical evidence, based on Evans function calculations, of the spectral stability of the profiles, independently of their shock strength. Finally, in a recent contribution \cite{FPZ-press}, we were able to prove analytically the spectral stability of dispersive shock profiles for the model with linear viscosity when the shock amplitude is sufficiently small. 

The purpose of this paper is to prove analytically that sufficiently weak dispersive shock profiles for the QHD model with nonlinear viscosity \eqref{QHD-E} are spectrally stable. In other words, our objective is to provide an analytical proof of the conjecture by Lattanzio and Zhelyazov \cite{LaZ21b} that dispersive shock profiles for model \eqref{QHD-E} are spectrally stable in the small-amplitude regime. Following \cite{LaZ21b}, we ought to define the linearized operator around the wave. Indeed, upon linearization of the system of equations \eqref{QHD-E} around the profile function we arrive at a natural spectral problem of the form
\[
\cL \begin{pmatrix}\hat{\rho}\\ \hat{u}\end{pmatrix} = \lambda \begin{pmatrix}\hat{\rho}\\ \hat{u}\end{pmatrix},
\]
where
\[
\cL\begin{pmatrix}\hat{\rho}\\ \hat{u}\end{pmatrix}
	\\:=\begin{pmatrix}
	s \hat{\rho}'  -(R \hat{u} + U \hat{\rho})'\\ $\,$ \\
	s \hat{u}' - (U \hat{u})' - (h'(R) \hat{\rho})' + \mu (R^{-1}(R \hat{u} + U \hat{\rho})')' - \mu (R^{-2} (R U)' \hat{\rho})' +\\
  	+ \tfrac{k^2}{2}(R^{-\frac{1}{2}}(R^{-\frac{1}{2}} \hat{\rho})'')' - \tfrac{k^2}{2}(R^{-\frac{3}{2}}(R^{\frac{1}{2}})'' \hat{\rho})'
	\end{pmatrix}
\]
is the linearized operator around both the density, $R= R(\xi)$, and the velocity, $U = U(\xi)$, profiles (for details, see \cite{LaZ21b} or Section \ref{secspectral} below). Here $' = d/d\xi$ denotes differentiation with respect to the Galilean variable of translation, $\xi = (x-st)/\ep$, and the variables $(\hat{\rho}, \hat{u})$ denote perturbations of the profile solution $(R,U)$. This operator acts on the space $L^2(\R) \times L^2(\R)$ of finite energy, localized perturbations. In lay terms and just like in the case of purely viscous shocks, spectral stability is the property that there is no spectrum of the linearized operator around the profile intersecting the complex half plane of complex values with positive real part, precluding the existence of solutions to the linearized evolution equations of the form $e^{\lambda t} (\hat{\rho}(\xi),\hat{u}(\xi))$ and with an explosive behaviour in time. Hence, our main result can be stated as follows (for its precise statement, see Theorem \ref{mainthm} below).

\begin{mtheorem}
Assume that $\mu, k, \ep > 0$ and $\gamma > 1$ are given. Suppose that the triplet $(R^\pm, U^\pm,s)$, with $R^\pm > 0$, defines a classical (Lax) shock of the underlying hyperbolic system in the absence of viscosity and dispersion (see system \eqref{eq:hyperbolic} below) and that $s \in (0, \bar{s})$, where the bound $\bar{s} > 0$ is defined in \eqref{sbound}. If the shock amplitude is sufficiently small, $|(R^+,U^+)-(R-,U^-)| \ll 1$, then the $L^2$-spectrum of the linearized operator $\cL$ is contained in the stable complex half plane, that is,
\[
\sigma(\cL)_{|L^2} \subset \{ \lambda \in \C \, : \, \Re \lambda < 0 \} \cup \{ 0 \}.
\]

\end{mtheorem}

A few remarks are in order. We prove that, in the small-amplitude (or weak shock) regime, profiles are monotone. This is a feature that is shared with purely viscous shock profiles. Our analysis exploits this property, as well as the smallness assumption on the shock amplitude, to establish energy estimates at the spectral level. For that purpose, we reformulate the problem in terms of new perturbation variables and transform the spectral problem into an equivalent one for a modified operator. The applied transformation is the composition of integration \cite{Go86,Go91}, a well-known technique to provide better energy estimates that profit from the compressivity of the shock, and a linear and invertible change of variables with the profiles as coefficients. Hence, as a result of the energy estimates, we identify an appropriate (and novel) weighted energy, defined in terms of the Bohmian potential and of the nonlinear viscosity, which allows us to conclude stability (see Remark \ref{remkeyenergy} below). Once again, the performed energy estimates, as well as the identification of the decaying energy, are convoluted (due to the nonlinearity of both viscosity and dispersion) and particular to the system under consideration. The condition on the shock speed, namely $s \in (0,\bar{s})$, not only guarantees that both end states are \emph{subsonic}, but also represents a restriction on the shock speed to conclude stability (see Remark \ref{remnonlineargood} for a discussion). Just like in the purely viscous case, there is accumulation of the of the essential spectrum near the origin, a fact that has to be taken into account in any (future) nonlinear stability analysis. To sum up, we provide for the first time the analytical proof of the spectral stability of sufficiently weak dispersive shock profiles for system \eqref{QHD-E}, a fundamental property that determines their nonlinear evolution.

\subsection*{Plan of the paper} This paper is structured as follows. Section \ref{secproperties} is devoted to write system \eqref{QHD-E} in conservation form and to describe the structure of dispersive shock profiles. The existence theory is reviewed and the monotonicity of sufficiently weak profiles is established. In Section \ref{secspectral} we pose the spectral stability problem and introduce a transformation of the perturbation variables suitable for our needs. The central and final Section \ref{sec:spectral} contains the energy estimate that controls the point spectrum, as well as the proof of our main result. We finish the paper with some concluding remarks and a discussion on open problems.

\subsection*{Notations}
Linear operators acting on infinite-dimensional spaces are indicated with calligraphic letters (e.g., $\cL$ and $\cA$). The domain of an operator, $\cL : X \to Y$, with $X$, $Y$ Banach spaces, is denoted as $D(\cL) \subseteq X$. We denote the real and imaginary parts of a complex number $\lambda \in \C$ by $\Re\lambda$ and $\Im\lambda$, respectively, as well as complex conjugation by ${\lambda}^*$. Standard Sobolev spaces of complex-valued functions on the real line will be denoted as $L^2(\R)$ and $H^m(\R)$, with $m \in \N$, endowed with the standard inner products and norms.

\section{Properties of dispersive shock profiles}
\label{secproperties}

In this section we recall the existence theory of \emph{dispersive shock profiles} for system \eqref{QHD-E} with nonlinear viscosity, which was mainly developed by Lattanzio and Zhelyazov in \cite{LaZ21a} (for related results, see \cite{Zhel-preprint}). In addition, we establish the monotonicity of sufficiently weak profiles, a property which will be useful for the stability analysis.

\subsection{Conservation form}
\label{secconserv}

First, we express system \eqref{QHD-E} in conservation form in the $(\rho, u)$ variables (see \cite{LaZ21a,LaZ21b}). 
To that end, we recall the definition of \emph{enthalpy} (see Gasser \cite{Gas01}),
\begin{equation}
\label{eq:ent}
	h(\rho) = \begin{dcases}
	\ln \rho, & \gamma = 1,\\
	\frac{\gamma}{\gamma-1}\rho^{\gamma-1}, & \gamma>1,
	\end{dcases}
\end{equation}
so that there holds the relation
\[
	p(\rho)_x = \rho h(\rho)_x.
\]
Apply this relation in order to recast system \eqref{QHD-E} into the following conservative form,
\begin{equation}
\label{QHD-EC}
\begin{aligned}
	\rho_t + (\rho u)_x &= 0,   \\
	u_t + \Big( \tfrac{1}{2}u^2 + h(\rho) \Big)_x & = \ep \mu \Big( \frac{(\rho u)_x}{\rho}\Big)_x + \ep^2 k^2 \left( \frac{(\sqrt{\rho})_{xx}}{\sqrt{\rho}}\right)_x.\\       
\end{aligned} 
\end{equation}

\begin{remark}
It is to be noticed that the conserved quantities in the QHD system \eqref{QHD-E} are the mass density and the velocity (not the momentum). The form of the system \eqref{QHD-EC} in Eulerian coordinates contrasts with its fluid dynamics counterpart with nonlinear viscosity and capillarity (see, e.g., \cite{DS85,HaLi96b,CCD15}), inasmuch as the conserved quantities are different for both systems, changing the structure of the equations. Indeed, the viscosity term for standard compressible fluids has the general form $(\tilde{\mu} (\rho) u_x)_x$, where $\tilde{\mu}(\rho)$ is a generic nonlinear viscosity term. Notice that, unlike the superfluid viscosity, it does not depend on $(\rho_x u/\rho)_x$. The dispersive term was already different due to the quantum Bohm potential, which contrasts with the capillarity terms of Korteweg type coming from interstitial work, having the form $k(\rho \rho_{xx} -\tfrac{1}{2}\rho_x^2)_x$ when the capillarity is assumed to be constant (see \cite{DS85,HaLi96b} for details).
\end{remark}

Once the system is recast in conservation form, one may examine the underlying \emph{hyperbolic} system, namely
\begin{equation}
\label{eq:hyperbolic}
	\begin{aligned}
		\rho_t+(\rho u)_x=0, \\
		u_t + \big(\tfrac{1}{2}u^2 + h(\rho) \big)_x=0.
	\end{aligned}
\end{equation}
The latter can be rewritten as $V_t + F(V)_x=0$, where $V=(\rho, u)^\top$ denotes the conserved quantities, and 
\[
F(V)=\left(\rho u,\tfrac{1}{2}u^2+h(\rho)\right),
\]
is the associated flux function. 
The system \eqref{eq:hyperbolic} is strictly hyperbolic; indeed, the Jacobian of $F$ is  
$$DF(\rho,u)=\begin{pmatrix}
u & \rho\\
h'(\rho) & u
\end{pmatrix},$$
and its real eigenvalues (characteristic speeds) are
\[
\lambda_1(\rho,u) := u - c_s(\rho) < \lambda_2(\rho,u) := u + c_s(\rho),
\]
where we have defined the sound speed as
\begin{equation}
\label{eq:sound-speed}
	c_s(\rho):=\sqrt{\rho h'(\rho)}=\sqrt{\gamma\rho^{\gamma-1}}.
\end{equation}
Notice that $c_s(\rho) > 0$ for each $\rho > 0$.

An ``inviscid'' shock is a planar front solution to the hyperbolic system \eqref{eq:hyperbolic} of the form
\begin{equation}
\label{hshock}
(\rho, u)(x,t) = \begin{cases}
(R^+, U^+), & x > st,\\
(R^-, U^-), & x < st,
\end{cases}
\end{equation}
where $s \in \R$ is the shock speed, and the constant end states satisfy $(R^+,U^+) \neq (R^-,U^-)$, $R^\pm > 0$, as well as the classical Rankine-Hugoniot jump conditions, namely
\begin{equation}
\label{RH}
\begin{aligned}
s(R^+-R^-) &= R^+ U^+ - R^- U^-,\\
s(U^+-U^-) &= \tfrac{1}{2}(U^+)^2 - \tfrac{1}{2}(U^-)^2 + h(R^+) - h(R^-).
\end{aligned}
\end{equation}

Planar shock fronts are weak solutions to the underlying hyperbolic system \eqref{eq:hyperbolic}. The shock speed $s \in \R$ is, of course, not arbitrary, but determined by the jump conditions \eqref{RH}. Furthermore, in order to guarantee uniqueness of the weak solution, an entropy condition should be imposed. In this paper we assume that the triplet $(R^\pm,U^\pm,s)$ satisfies \emph{Lax entropy condition} \cite{Da4e,La57}: it is said that the triplet is a Lax $k$-shock, with $k = 1,2$, provided that
\begin{equation}
\label{Lax}
\lambda_k(R^+,U^+) < s < \lambda_k(R^-,U^-).
\end{equation}

The end state $(R^\pm,U^\pm)$ is said to be subsonic (respectively, supersonic) if $|U^\pm| < c_s(R^\pm)$ (respectively, $|U^\pm| > c_s(R^\pm)$). In the case where $|U^\pm| = c_s(R^\pm)$ the state is said to be sonic.

\begin{remark}
Notice that if $R^+ = R^- > 0$ then conditions \eqref{RH} imply that $U^+ = U^-$. Therefore, in the case of a shock front with different end states we necessarily have that $R^+ \neq R^-$ and we may, without loss of generality, consider $|R^+ - R^-|$ as a measure of the shock amplitude.
\end{remark}

\subsection{Dispersive shock profiles}
\label{secshocks}

Assume that the triplet $(R^\pm,U^\pm,s)$ defines a planar shock front for system \eqref{eq:hyperbolic} that satisfies Rankine-Hugoniot conditions \eqref{RH} as well as Lax entropy condition \eqref{Lax}. Dispersive shock profiles for the system \eqref{QHD-EC} with viscosity and dispersion are traveling wave solutions of the form
\begin{equation}
\label{eq:prof}
	\rho(x,t)=R \Big( \frac{x-st}{\ep} \Big),\qquad \qquad u(t,x)=U\Big( \frac{x-st}{\ep} \Big),
\end{equation}
for a given constant $\ep > 0$ and with asymptotic limits
\[
	R^{\pm}=\lim_{\xi \rightarrow \pm \infty}R(\xi), \qquad \qquad  U^{\pm}=\lim_{\xi \rightarrow \pm \infty}U(\xi).
\]

Hence, for small values of $\ep > 0$, shock profile solutions to \eqref{QHD-E} constitute approximations of the discontinuous front \eqref{hshock}. Upon substitution of \eqref{eq:prof} into \eqref{QHD-E} one obtains the system
\begin{equation}
\label{profsyst}
\begin{aligned}
-sR'+(R U)'&=0,\\
-s U' + \tfrac{1}{2}(U^2)' + h(R)' &= \mu \Big(\frac{(R U)'}{R}\Big)' + k^2 \Big(\frac{(\sqrt{R})''}{\sqrt{R}}\Big)',
\end{aligned}
\end{equation}
where $' = d/d\xi$ denotes differentiation with respect to the Galilean variable $\xi = (x-st)/\ep$. Once the physical constant $\ep > 0$ has been fixed, with a slight abuse of notation but with no loss of generality, we make the transformation,
\begin{equation}
\label{galvar}
x \, \rightarrow \frac{x-st}{\ep},
\end{equation}
so that now, and for the rest of the paper, $x$ denotes the Galilean variable of translation.

Let us make some rearrangements of the profile equations that will be useful later on. Substitute the first equation in \eqref{profsyst} into the second to obtain
\begin{equation}
\label{profsyst2}
\begin{aligned}
-sR'+(R U)'&=0,\\
-s U' + \tfrac{1}{2}(U^2)' + h(R)' &= \mu \Big(\frac{sR'}{R}\Big)' + k^2 \Big(\frac{(\sqrt{R})''}{\sqrt{R}}\Big)'.
\end{aligned}
\end{equation}
Integrating the first equation in \eqref{profsyst2} up to $\pm \infty$ yields
\begin{equation}
\label{eq:U-R}
	U = s - \frac{A}{R},
\end{equation}
for all $x \in \R$, where $A := (s-U^{\pm})R^{\pm}$ is a constant in view of Rankine-Hugoniot conditions \eqref{RH}. Similarly, integration of the second equation in \eqref{profsyst2} implies the relation
\[
-sU + \tfrac{1}{2}U^2 + h(R) = s \mu \frac{R'}{R} + k^2 \frac{(\sqrt{R})''}{\sqrt{R}} + B,
\]
where $B:= -s U^\pm + \tfrac{1}{2}(U^\pm)^2 + h(R^\pm)$ is a constant of integration. 
Hence, we may use \eqref{eq:U-R} in order to obtain a planar ODE for the profile $R$.
The result is
\begin{equation}\label{2Dsys}
	R'' = \frac{2}{k^2}f(R) - \frac{2 s \mu}{k^2} R' + \frac{(R')^2}{2 R},
\end{equation}
where
\[
f(R) := R h(R) + \frac{A^2}{2R} - \tfrac{1}{2} s^2 R - RB.
\]
We can express the constants $A$ and $B$ in terms of $R^\pm$ alone with the help of Rankine-Hugoniot conditions \eqref{RH} and some straightforward algebra. This yields the expression
\begin{equation}
\label{longf}
\begin{aligned}
f(R) =\frac{R}{R^+ + R^-} &\left[\frac{(R^+R^-)^2}{R^2} \Big(\frac{h(R^+)-h(R^-)}{R^+ - R^-}\Big) + (R^+ + R^-) h(R)  \right. \\
& \left. \quad -\frac{(R^+)^2 h(R^+)-(R^-)^2 h(R^-)}{R^+ - R^-}\right].
\end{aligned}
\end{equation}

Henceforth, for a given triplet $(R^\pm, U^\pm,s)$ defining a Lax shock, it is sufficient to solve equation \eqref{2Dsys} for the density profile in the $(R, R')$-phase plane in order to obtain a dispersive shock profile. The velocity profile $U$ is then retrieved via equation \eqref{eq:U-R}. This is precisely the strategy  of proof by Lattanzio and Zhelyazov \cite{LaZ21a}, whose main existence result for dispersive shock profiles in can be stated as follows.
\begin{theorem}[existence of dispersive shocks]
\label{theoex}
Suppose that the triplet $(R^\pm,U^\pm,s)$ satisfies Rankine-Hugoniot jump conditions \eqref{RH} with $R^\pm > 0$ and defines either
\begin{itemize}
\item[\rm{(i)}] a Lax 2-shock with a subsonic right state (i.e. with $|U^+| < c_s(R^+)$); or,
\item[\rm{(ii)}] a Lax 1-shock with a subsonic left state (i.e. with $|U^-| < c_s(R^-)$).
\end{itemize}
Then there exists a dispersive shock profile solution $(R(x), U(x))$, $x \in \R$, connecting $(R^-,U^-)$ to $(R^+,U^+)$.
\end{theorem}
\begin{proof}
See Lemma 1 and Corollary 2 in \cite{LaZ21a}.
\end{proof}

\begin{remark}
Theorem \ref{theoex} guarantees the existence of a dispersive shock profile independently of the shock amplitude, that is, of $|R^+-R^-|$. The density profile $R = R(x)$ may be non-monotone depending on the magnitude of the ratio $\mu / k$. Yet, it stays away from vacuum, inasmuch as $R(x) \geq C > 0$ for all $x \in \R$ (see Lattanzio and Zhelyazov \cite{LaZ21a}); one can fix the end state $R^-$, as well as the physical parameters, and vary $R^+$ in order to examine the behavior of the profiles as the amplitude $\vep = |R^+ - R^-|$ increases. Figure \ref{figprofiles} depicts the numerical calculation of different density profiles $R = R(x)$ for fixed parameter values $\gamma = 3/2$, $R^-=0.7$, $s = 1$, $\mu = 1$ and $k = \sqrt{2}$. The profiles are numerically computed as solutions to the profile equation \eqref{2Dsys} for different values of the shock amplitude, $\vep = R^- - R^+$. Notice that for small values of $\vep$ the profiles are monotone, but they start to oscillate as $\vep$ increases; for further details, see \cite{LaZ21a,Zhel-preprint}. 
\begin{figure}[t]
\begin{center}
\includegraphics[scale=.55, clip=true]{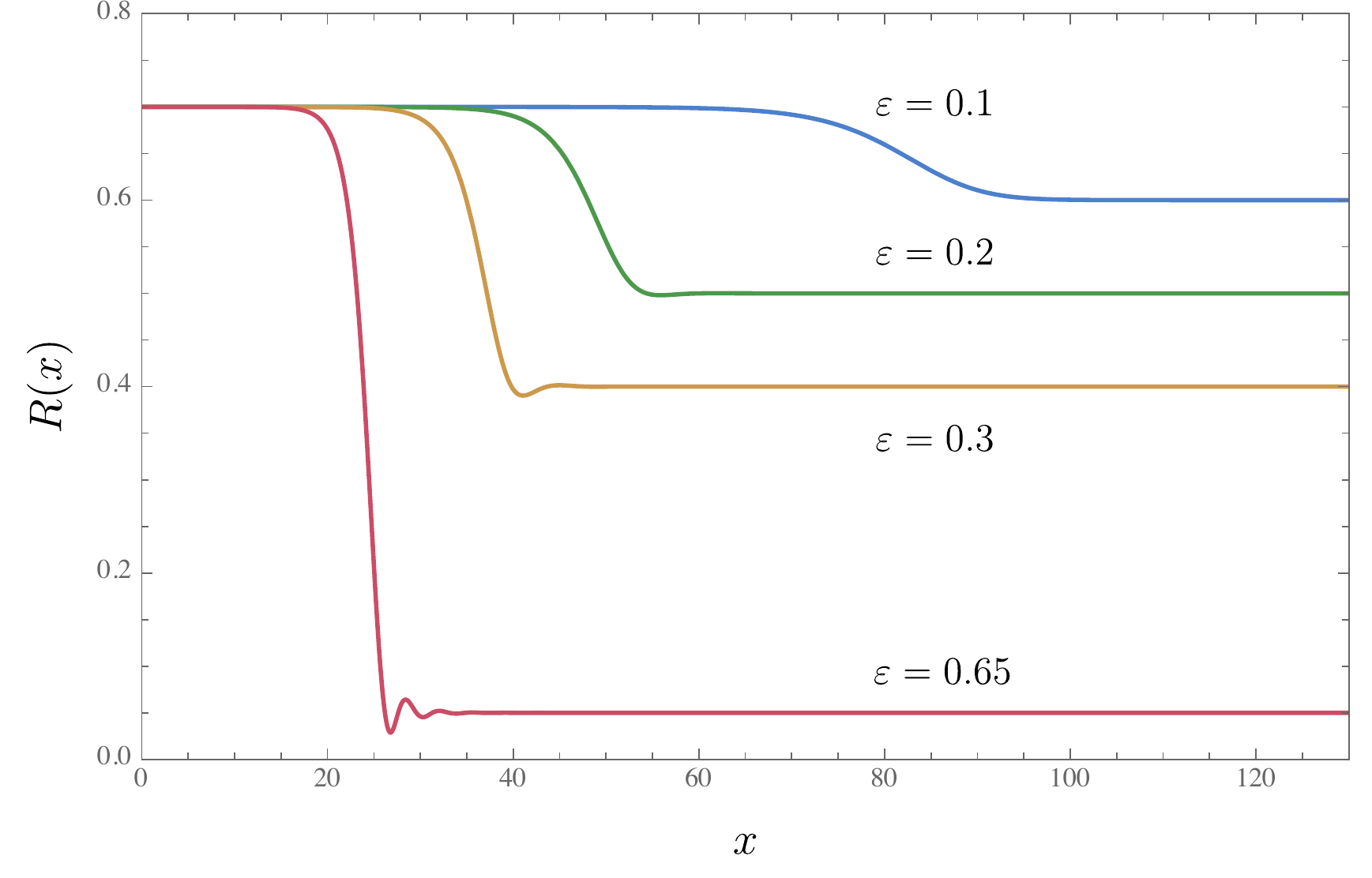}
\end{center}
\caption{\small{Numerical calculation of the density profiles $R = R(x)$ as solutions to the equation \eqref{2Dsys} with parameter values $\gamma = 3/2$, $R^- = 0.7$, $s=1$, $\mu = 1$, $k = \sqrt{2}$ for different values of the shock strength $\varepsilon = R^- - R^+$ (color online).}}
\label{figprofiles}
\end{figure}
\end{remark}

\subsection{Monotonicity of small-amplitude shock profiles}
\label{sec:monotonicity}

In the sequel, we focus our attention on the case of \emph{weak} shock profiles, for which the shock amplitude $\e:=|R^+-R^-|$ is sufficiently small. In the small-amplitude regime dispersive shocks comply with compressivity, that is, they are monotone and resemble their purely viscous profile counterparts. The monotonicity of the profiles is a crucial property that plays a key role in the proof of spectral stability (see Section \ref{sec:spectral} below). In this section we show that sufficiently weak dispersive shocks are indeed monotone and establish some useful estimates for its derivatives.

\begin{proposition}[monotonicity of weak profiles]
\label{lem:R-prop}
Let $R^->0$.
There exists a sufficiently small $\varepsilon_0>0$ such that:
\begin{itemize}
\item[\rm{(a)}] If $s>0$ and $0<R^--R^+<\varepsilon_0$, then there exists a heteroclinic trajectory $R$ of \eqref{2Dsys} connecting $(R^-, 0)$ to $(R^+, 0)$ (equilibria in the $(R, R')$-phase plane), which is monotone decreasing with $R'(x)<0$ for all $x\in\R$.
\item[\rm{(b)}] If $s<0$ and $0<R^+-R^-<\varepsilon_0$ then there exists a heteroclinic trajectory $R$ of \eqref{2Dsys} connecting $(R^-, 0)$ to $(R^+, 0)$ (equilibria in the $(R, R')$-phase plane), which is monotone increasing with $R'(x)>0$ for all $x\in\R$.
\end{itemize}
In both cases, we have 
\begin{equation}
\label{eq:prop-R}
	|R'(x)|\leq C\e^2, \qquad |R''(x)|\leq C\e|R'(x)|, 
\end{equation}
for all $x \in \R$ and for some uniform constant $C > 0$.
\end{proposition}

\begin{proof}
Let us start by proving \textrm{(a)}. In this case we have $s>0$ and assume that the amplitude of the shock, $\e:= R^- - R^+$,
satisfies $\e\in(0,\e_0)$ with $\vep_0 > 0$ to be detrmined.
Since  $R^+ = R^- - \e$, substitution into \eqref{longf} yields
\begin{align*}
	f(R)&=\frac{R}{2 R^- - \e}\left[\frac{((R^- - \e) R^-)^2}{R^2} \Big( \frac{h(R^-) - h(R^- - \e)}{\e} \Big) + (2 R^- - \e) h(R) \right.\\
	&\qquad \qquad\qquad  \qquad \left.- \frac{(R^-)^2 h(R^-) - (R^- - \e)^2 h(R^- - \e)}{\e} \right]\\
	&=\frac{m_1}{R}+Rh(R)-Rm_2,
\end{align*}
where
\begin{align*}
	m_1&:= \frac{((R^- - \e) R^-)^2}{2 R^- - \e} \left(\frac{h(R^-) - h(R^- - \e)}{\e} \right),\\
	m_2&:= \frac{1}{2 R^- - \e} \left( \frac{(R^-)^2 h(R^-) - (R^- - \e)^2 h(R^- - \e)}{\e} \right),
\end{align*}
are constants depending only on $R^-$ and $\e$. A direct differentiation yields
\begin{align*}
	f'(R)&= \begin{dcases}
	-\frac{m_1}{R^2} + \ln R + 1 - m_2, & \gamma =1,\\
	-\frac{m_1}{R^2} + \frac{\gamma^2}{\gamma - 1} R^{\gamma - 1} - m_2, &\gamma > 1,
\end{dcases}\\
	f''(R)&= \begin{dcases}
	\frac{2 m_1}{R^3} + \frac{1}{R}, \qquad \qquad & \gamma = 1,\\
	\frac{2 m_1}{R^3} + \gamma^2 R^{\gamma - 2}, & \gamma > 1.
	\end{dcases}
\end{align*}
As a consequence, $f'(R) \rightarrow -\infty$ as $R \rightarrow 0^+$ and $f'(R) \rightarrow \infty$ as $R \rightarrow \infty$, because of $m_1>0$.
Moreover, $f''(R) >0$ for $R > 0$ and so, $R^{\pm}$ are the only two positive zeros of $f$. 
Furthermore, $f'(R)$ is monotonically increasing and it has a unique zero $R_0 = R_0(\e)$ which, in addition, verifies $R^+ < R_0 < R^-$.
	
Since the function $f$ is not defined for $\e=0$, we now introduce a function $\overline{f}$ which is defined in a neighborhood of $\e=0$ and agrees with $f$ for $\e>0$. 
Then, we apply the implicit function theorem to $\overline{f}$ to study the behavior of $R_0(\e)$ for $\e > 0$ small.
First, choose $\tilde{\e}_0 \in (0,R^-)$, and let us define for $|\nu| < \tilde{\e}_0$,
\begin{equation*}
	S_1(\nu):= \begin{dcases}
	\frac{h(R^-) - h(R^- - \nu)}{\nu}, & \nu \neq 0,\\
	h'(R^-), & \nu = 0,
	\end{dcases}
\end{equation*}
and
\begin{equation*}
	S_2(\nu):= \begin{dcases}
	\frac{(R^-)^2h(R^-) - (R^- - \nu)^2 h(R^- - \nu)}{\nu}, & \nu \neq 0,\\
	R^-(2 h(R^-) + R^- h'(R^-)), & \nu = 0.
	\end{dcases}
\end{equation*}
Next, define on $\mathbb{R}^+ \times \{ |\nu|  < \tilde{\e}_0 \}$,
\begin{equation*}
	\overline{f}(R, \nu) = \frac{R}{2 R^- - \nu} \left[\frac{((R^- - \nu) R^-)^2}{R^2}S_1(\nu) + (2 R^- - \nu)h(R) - S_2(\nu)\right].
\end{equation*}
Clearly, $\overline{f}$ is defined in a neighborhood of $(R^-,0)$, it is smooth and $\overline{f}(R,\e) = f(R)$ for all $R > 0$ and all $\e \in (0,\tilde{\e}_0)$. Let $\psi(R,\nu):= \overline{f}_R(R,\nu)$, the derivative of $\overline{f}$ with respect to $R$. Hence $\psi(R^-,0) = 0$ and
\[
	 \psi_{R}(R^-,0) = \overline{f}_{RR}(R^-,0) = 3 h'(R^-) + R^- h''(R^-) = \gamma(\gamma + 1)(R^-)^{\gamma - 2}.
\]
Therefore, $\psi_R(R^-,0) > 0 $ for $\gamma \geq 1$. By the implicit function theorem there exist open intervals $I$ and $J$, with $(R^-,0) \in I \times J \subset \mathbb{R}^+ \times \{ |\nu|  < \tilde{\e}_0 \}$, and a unique function $\widetilde{R}_0: J \rightarrow I$ such that $\psi(\widetilde{R}_0(\nu),\nu) = 0$ for $\nu \in J$. In particular $\widetilde{R}_0(0) = R^-$ and $\widetilde{R}_0$ is smooth. We also have
\[
	\begin{aligned}
	\psi_\nu (R^-,0) &= \frac{1}{2}(3 h'(R^-) + R^- h''(R^-)),\\
	\psi_{RR}(R^-,0) &= -\frac{3 h'(R^-)}{R^-} + 3 h''(R^-) + R^- h'''(R^-),\\
	\psi_{R \nu}(R^-,0) &= -\frac{1}{2}\left(\frac{3 h'(R^-)}{R^-} + h''(R^-)\right),\\
	\psi_{\nu \nu}(R^-,0) &= -\frac{1}{6}\left(\frac{3 h'(R^-)}{R^-} + 9 h''(R^-) + 2 R^- h'''(R^-)\right).
	\end{aligned}
\]
By differentiating $\psi(\widetilde{R}_0(\nu),\nu) = 0$, we infer
\begin{align*}
	\widetilde{R}_0'(0)&= - \frac{\psi_\nu(R^-,0)}{\psi_R(R^-,0)}, \\ 
	\widetilde{R}_0''(0)&= \left. \left( \frac{- (\psi_R)^2 \psi_{\nu\nu} + 2 \psi_{\nu} \psi_R \psi_{\nu R} - (\psi_{\nu})^2 \psi_{RR}}{(\psi_R)^3} \right) \right|_{(R^-,0)}.
\end{align*}
This yields, in turn, $\widetilde{R}_0'(0) = -\tfrac{1}{2}$ and $\widetilde{R}_0''(0) = (\gamma - 3)/12 R^-$. Since $\widetilde{R}_0(\nu) = R_0(\nu)$ for $\nu \in J \cap \{\nu > 0\}$, we obtain
\begin{equation}
\label{expansion_R0}
	R_0(\e) = R^- - \frac{\e}{2} + \frac{\gamma - 3}{24 R^-}\e^2 + O(\e^3).
\end{equation}
Let us expand $f(R)$ around $R_0$:
\begin{equation}
\label{expansion_f}
	f(R) = f(R_0) + f'(R_0)(R-R_0) + \frac{1}{2} f''(R_0)(R-R_0)^2 + O(|R-R_0|^3), 
\end{equation}
and consider the change of variables
\begin{equation}
\label{expr_R}
	R = \e \overline{R} + R^- - \frac{\e}{2}.
\end{equation} 
In this fashion, the value $R = R^+$ corresponds to $\overline{R} = -\frac{1}{2}$ and $R = R^-$ is mapped into $\overline{R} = \frac{1}{2}$. 
Therefore, the interval $[R^+, R^-]$ is in one-to-one correspondence with $\left[-\frac{1}{2},\frac{1}{2}\right]$ in the variable $\overline{R}$ and it is independent of $\e$.

Now, let us make a change of the independent variable, $z := \varepsilon x$, and substitute \eqref{expr_R} into \eqref{2Dsys}. This yields
\begin{equation}\label{2Dsys_Rbar}
	\varepsilon^3 \overline{R}_{zz} = \frac{2}{k^2} f\left(\e \overline{R} + R^- - \frac{\e}{2}\right)-\frac{2 s \mu}{k^2}\e^2 \overline{R}_z + 
	\frac{1}{2}\frac{\e^4}{\e\overline{R} + R^- -\frac{\e}{2}}(\overline{R}_z)^2.
\end{equation}
Using \eqref{expansion_R0} we compute
\begin{equation}
\label{exprf}
\begin{aligned}
f(R_0) &= - \tfrac{1}{8}\gamma (\gamma + 1) (R^-)^{\gamma - 2} \e^2 + O(\e^3),\\
f'(R_0) &= 0,\\
f''(R_0) &= \gamma (\gamma + 1) (R^-)^{\gamma - 2} + O(\e). 
\end{aligned}	
\end{equation}
From \eqref{expr_R} and \eqref{expansion_R0} one clearly has $R - R_0 = \e \overline{R}+ O(\e^2)$. Hence, upon substitution of \eqref{exprf} into \eqref{expansion_f}, we end up with 
\[
	f(R) = -\tfrac{1}{8}{\gamma (\gamma + 1)} (R^-)^{\gamma - 2} \e^2 + \tfrac{1}{2}{\gamma (\gamma + 1)} (R^-)^{\gamma - 2} \e^2 \overline{R}^2 + O(\e^3).
\]
Upon substitution of last expression into \eqref{2Dsys_Rbar}, we obtain the leading order equation,
\begin{equation}
	\overline{R}_z = c \Big(\overline{R}^2 - \tfrac{1}{4}\Big), \label{reduced_eq}
\end{equation}
where
\begin{equation*}
	c = \frac{\gamma (\gamma + 1) (R^-)^{\gamma - 2}}{2 s \mu} > 0.
\end{equation*}
Equation \eqref{reduced_eq} has solutions converging to $\mp 1/2$ as $z \rightarrow \pm \infty$, with $R_z< 0$, and which are contained in the interval $(-\frac{1}{2},\frac{1}{2})$.
	
Now, let $Q = \overline{R}_z$. The slow manifold (see, e.g., Jones \cite{J95}) in the phase plane for the $(\overline{R},Q)$ variables, is given by
$\cM_{\e} = \left\{(\overline{R},Q) \in \R^2 \, : \, Q = h^{\varepsilon}(\overline{R}) \right\}$, where $h^{\e}(\overline{R}) = h_0(\overline{R}) + \e h_1(\overline{R},\e)$, $h_0(\overline{R}) = c \big( \overline{R}^2 - \tfrac{1}{4} \big)$ and $h_1$ is a smooth function of $\vep$ and $\overline{R}$ of order $O(1)$.
Therefore, the equation on the slow manifold is
\begin{equation}
\label{eq_slow_manifold}
	\overline{R}_z= c \Big(\overline{R}^2 - \tfrac{1}{4}\Big) + \varepsilon h_1(\overline{R}, \e).
\end{equation}
For sufficiently small values of $\e > 0$, the term $\varepsilon h_1(\overline{R}, \e)$ does not affect the monotonicity of the solutions to \eqref{eq_slow_manifold}, which converge to $\mp 1/2$ as $z \rightarrow \pm \infty$. 
Moreover, these solutions are strictly decreasing and we have
\[
	R(x) = \varepsilon \overline{R}(\e x) + \tfrac{1}{2}(R^+ + R^-), \qquad x\in\R.
\]
Direct differentiation with respect to $x$ yields
\begin{equation}
\label{eq_dR}
	R'(x) = \e^2 \overline{R}'(\e x).
\end{equation}	
Thus, we deduce the existence of $\vep_0 > 0$ sufficiently small such that for all $\vep \in (0, \vep_0)$, we have $R'(x) < 0$ for all $x \in \R$. Moreover, since $|\overline{R}\,'(z)| \leq C$, where $C$ does not depend on $\e$, we obtain the first inequality in \eqref{eq:prop-R}.
Furthermore, direct differentiation of \eqref{eq_dR} with respect to $x$ and equation \eqref{eq_slow_manifold}, yield $R''(x) = \e R'(x)\left(2 c \overline{R}(\e x)  + \e \partial_R h_1(\overline{R}(\e x), \e)\right)$. Since $\big| 2 c R + \e \partial_R h_1 \big| \leq C$ for some uniform constant $C>0$ independent of $\vep$, we obtain the second inequality in \eqref{eq:prop-R}. This shows (a).
	
We now prove \textrm{(b)}. Recall that now we have $s<0$, $R^- < R^+$ and $\e := R^+ - R^- > 0$ is the amplitude of the shock. Let us make the change of variables $\tilde{x} = - x$ and $\widetilde{R}(\tilde{x}) = R(x)$ in \eqref{2Dsys}. Denoting $' = d/d\tilde{x}$, we obtain the equation
\begin{equation*}
	\widetilde{R}''=\frac{2}{k^2} f(\widetilde{R}) - \frac{2 \tilde{s} \mu}{k^2} \widetilde{R}' + \frac{\widetilde{R}'^2}{2 \widetilde{R}},
\end{equation*}
where $\tilde{s} = -s$. Notice that this equation has the same form as equation \eqref{2Dsys}. Let $\widetilde{R}^+ = R^-$ and $\widetilde{R}^- = R^+$. Hence, we repeat the steps in case \textrm{(a)} (with the difference that we now expand around $\widetilde{R}^+$) in order obtain the same reduced equation, namely, $\overline{R}_{\tilde z}= c \big( \overline{R}^2 - \tfrac{1}{4} \big)$, where $\tilde z=\e\tilde x$, but now with 
\[
c=\frac{\gamma (\gamma + 1) (\widetilde{R}^+)^{\gamma - 2}}{2 \tilde{s} \mu} > 0.
\]
The conclusion follows exactly as in case \textrm{(a)}.
\end{proof}
\begin{remark}
\label{remonlyspos}
From the statement of Proposition \ref{lem:R-prop} we distinguish two cases: either $s>0$, and we have a monotone decreasing density profile $R$; or, $s<0$, and the profile $R$ is strictly increasing. For simplicity in the exposition, from this point on and for the rest of the paper we consider only the case $s>0$ (monotone decreasing density profile). Our results can be easily extended to case $s<0$.
\end{remark}
\begin{remark}
From relation \eqref{eq:U-R} between the two profile functions, one can obtain equivalent estimates for the velocity profile. For instance, let us express the constant $A$ as a function of $R^+$ and $R^-$,
\[
A = R^+ R^- \sqrt{2 \frac{h(R^+) - h(R^-)}{(R^+)^2 - (R^-)^2}.}
\]
If we choose $R^+ = R^- -\e$, with $0 < \e < R^-$ (recall we are now assuming $s > 0$), then we have
\begin{equation*}
A = \frac{( R^- - \e)R^-}{\sqrt{2 R^- - \e}}\sqrt{2 \frac{h(R^-) - h(R^- - \e)}{\e}}.
\end{equation*}
Hence, there exists a constant $\e_0 > 0$ such that the following expansion holds
\begin{equation}
	A = R^- c_s(R^-) - \tfrac{1}{4}(\gamma + 1) c_s(R^-) \e + O(\e^2), \label{expansion_A}
\end{equation}
for $0 < \e < \e_0$, where the sound speed at the end state is $c_s(R^-) = \sqrt{\gamma (R^-)^{\gamma - 1}}$. Therefore, by differentiating and using \eqref{eq:prop-R} and \eqref{expansion_A}, we obtain
\begin{equation}
\label{eq:est-U}
	|U'(x)|\leq C\e^2, \qquad |U''(x)|\leq C\e|R'(x)|,	
\end{equation}
for any $x\in\R$, for some constant $C>0$ independent of $\e$. 
\end{remark}

\section{The stability problem}
\label{secspectral}

In this section we describe the spectral stability problem associated to a dispersive shock profile for system \eqref{QHD-E}. Moreover, we introduce a new transformation in the perturbation variables that allows to perform the energy estimates.

\subsection{Perturbation equations and the linearized operator}

To begin our stability analysis, let us consider a solution to \eqref{QHD-EC} of the form $(\widetilde{\rho},\widetilde{u}) + (R,U)$ where $(\widetilde{\rho},\widetilde{u})$ is a perturbation of a given dispersive shock profile $(R,U)$. As before, we make the transformation \eqref{galvar} and, in addition, rescale the time variable as $t \rightarrow t/\ep$. Hence, one finds that the perturbation variables solve the nonlinear system
\[
\begin{aligned}
	-s \widetilde{\rho}_x + \widetilde{\rho}_t - s R_x+ \big((\widetilde{\rho}+R) (\widetilde{u}+U)\big)_x &= 0,   \\
	-s \widetilde{u}_x + \widetilde{u}_t -s U_x + \Big( \tfrac{1}{2}(\widetilde{u}+U)^2 + h(\widetilde{\rho}+R) \Big)_x & = 
	\mu \left( \frac{((\widetilde{\rho}+R) (\widetilde{u}+U))_x}{\widetilde{\rho}+R}\right)_x \\
	&\qquad  + k^2 \left( \frac{\big(\sqrt{\widetilde{\rho}+R \,} \,\big)_{xx}}{\sqrt{\widetilde{\rho}+R\,}}\right)_x.\\       
\end{aligned} 
\]
For nearby perturbations the leading approximation is given by the linearization of the former system around $(R,U)$, yielding the following linear system of equations
\[
\begin{aligned}
	\widetilde{\rho}_t &= s\widetilde{\rho}_x - \big( R \widetilde{u} + U \widetilde{\rho} \big)_x ,   \\
	\widetilde{u}_t & = s \widetilde{u}_x - \big(U \widetilde{u}\big)_x - \big(h'(R) \widetilde{\rho}\big)_x + \mu \big( R^{-1} ( R \widetilde{u} + U \widetilde{\rho} )_x\big)_x - \mu \big( R^{-2} (RU)_x \widetilde{\rho} \big)_x +\\
	&\quad + \tfrac{k^2}{2}  \big( R^{-\frac{1}{2}} (R^{-\frac{1}{2}} \widetilde{\rho})_{xx})\big)_x - \tfrac{k^2}{2} \big( R^{-\frac{3}{2}} (R^{\frac{1}{2}})_{xx}) \widetilde{\rho}\big)_x ,\\   
\end{aligned}
\]
where we have substituted the profile equations \eqref{profsyst2}. Specializing these linearized equations to perturbations of the form $(\widetilde{\rho},\widetilde{u})(x,t) = e^{\lambda t} (\hat{\rho},\hat{u})(x)$, where $\lambda \in \C$ and $(\hat{\rho},\hat{u})$ lies in an appropriate Banach space $X$, we naturally arrive at the eigenvalue problem
\begin{equation}
\label{eq_variable_coeff}
	\cL\begin{pmatrix} \hat{\rho}\\
	\hat{u} \end{pmatrix} = \lambda \begin{pmatrix} \hat{\rho}\\ \hat{u}
	\end{pmatrix},
\end{equation}
where the operator $\cL$ has the form
\begin{equation}
\label{operator_Lnd}
	\cL\begin{pmatrix}\hat{\rho}\\ \hat{u}\end{pmatrix}
	\\:=\begin{pmatrix}
	s \hat{\rho}'  -(R \hat{u} + U \hat{\rho})'\\ $\,$ \\
	s \hat{u}' - (U \hat{u})' - (h'(R) \hat{\rho})' + \mu (R^{-1}(R \hat{u} + U \hat{\rho})')' - \mu (R^{-2} (R U)' \hat{\rho})' +\\
  	+ \tfrac{k^2}{2}(R^{-\frac{1}{2}}(R^{-\frac{1}{2}} \hat{\rho})'')' - \tfrac{k^2}{2}(R^{-\frac{3}{2}}(R^{\frac{1}{2}})'' \hat{\rho})'
	\end{pmatrix}.
\end{equation}
Here, once again, $' = d/dx$ denotes differentiation with respect to the Galilean translation variable. 

Motivated by the notion of spatially localized, finite energy perturbations in the Galilean coordinate frame in which the profile is stationary, we consider the (complex) space $X = L^2(\R) \times L^2(\R)$ and regard $\cL$ as a closed, densely defined operator acting on $L^2(\R) \times L^2(\R)$ with domain $\cD(\cL) := H^3(\R) \times H^2(\R)$. Formally, a necessary condition for the profile to be stable is the absence of solutions $(\hat{\rho}, \hat{u}) \in \cD(\cL)$ to the spectral problem \eqref{eq_variable_coeff} for some $\lambda \in \C$ with $\Re \lambda > 0$, precluding the existence of solutions to the linearized equations of the form $e^{\lambda t} (\hat{\rho}, \hat{u})(x)$ that grow exponentially in time. Hence, we have the following
\begin{definition}[spectral stability]
A dispersive shock profile $(R,U)$ is \emph{spectrally stable} if the $L^2$-spectrum of the linearized operator $\cL$ is contained in the stable complex half plane, that is,
\[
\sigma(\cL) \subset \{ \lambda \in \C \, : \, \Re \lambda < 0 \} \cup \{ 0 \}.
\]
\end{definition}

\begin{remark}
\label{remspectra}
For any closed, densely defined operator $\cL$, the standard definitions of its spectrum and its resolvent set (denoted as $\sigma(\cL)$ and $\varrho(\cL)$, respectively) apply here (see, e.g., Kato \cite{Kat80}). For nonlinear wave stability purposes (see, e.g., Kapitula and Promislow \cite{KaPro13}), we employ Weyl's partition of the spectrum \cite{We10}: the point spectrum $\ptsp(\cL)$ is the set of all $\lambda \in \C$ such that $\cL- \lambda$ is Fredholm with zero index and non-trivial kernel; the essential specrtum, denoted as $\ess(\cL)$, is the set of all complex values $\lambda$ for which either $\cL - \lambda$ is not Fredholm, or it is Fredholm with non-zero index. In the case of a closed, densely defined operator it can be shown that $\sigma(\cL) = \ptsp(\cL) \cup \ess(\cL)$ (see, e.g., \cite{Kat80}).
This partition is useful for stability analyses because $\ess(\cL)$ is easy to compute and $\ptsp(\cL)$ comprises discrete isolated eigenvalues with finite multiplicities (see \cite{KaPro13}, chapter 2). As mentioned before, since we are computing the spectrum of $\cL$ with respect to the space $L^2 \times L^2$, our stability analysis pertains to localized, finite energy perturbations.
\end{remark}

\begin{remark}
\label{remzeroevalue}
As expected, $\lambda = 0$ belongs to $\ptsp(\cL)$. This eigenvalue is associated to the eigenfunction $(R',U')$. This follows because differentiation of the profile equations \eqref{profsyst2} yields $\cL (R', U')^\top = 0$. In this case $\lambda = 0$ is called the translation eigenvalue.
\end{remark}

By taking asymptotic limits when $x \to \pm \infty$ in the expressions for the coefficients of $\cL$ we obtain the asymptotic operators at the end states,
\begin{equation}
\label{asymop}
	\cL^\pm \begin{pmatrix}\hat{\rho}\\ \hat{u}\end{pmatrix}
	\\:=\begin{pmatrix}
	(s-U^\pm) \hat{\rho}'  - R^\pm \hat{u}'\\ $\,$ \\
	(s-U^\pm) \hat{u}' - h'(R^\pm) \hat{\rho}' + \mu ( \hat{u}'' + U^\pm (R^\pm)^{-1} \hat{\rho}'') + \tfrac{1}{2} k^2 (R^\pm)^{-1} \hat{\rho}'''
	\end{pmatrix},
\end{equation}
with $\cL^\pm : \cD(\cL)  \subset L^2(\R) \times L^2(\R) \to L^2(\R) \times L^2(\R)$. From standard stability theory, the asymptotic operators determine the algebraic curves in the complex plane that bound the essential spectrum of the original operator $\cL$ (for details, see \cite{KaPro13,San02}).

\subsection{New perturbation variables}

In this section we recast the spectral problem \eqref{eq_variable_coeff} in terms of new perturbation variables. First, we follow the fundamental observation by Goodman \cite{Go86,Go91} and use \emph{integrated variables}. For that purpose, consider
\begin{equation}
\label{intvar}
\rho(x) := \int_{-\infty}^x \hat{\rho}(\xi) \, d\xi, \qquad u(x) :=\int_{-\infty}^x\hat{u}(\xi) \, d\xi.
\end{equation}
This transformation removes the zero eigenvalue without any further modification of the spectrum and, notably, provides better energy estimates. Integrating the spectral equation \eqref{eq_variable_coeff} for any $\lambda \neq 0$ we observe that the integrated perturbation variables $\rho$ and $u$ decay exponentially as $|x| \rightarrow \infty$. Thus, expressing the spectral problem in terms of $\rho$ and $u$, and integrating \eqref{eq_variable_coeff} from $-\infty$ to $x$, we get the following equivalent system
\begin{align}
\lambda\rho &= (s-U)\rho'-Ru', \label{eq:integrated-rho}\\
\lambda u &= -h'(R)\rho'+(s-U)u'+\mu R^{-1}(Ru'+U\rho')' - \mu R^{-2}(RU)'\rho' + \nonumber\\ &\qquad + \tfrac{1}{2} k^2 R^{-\frac12}(R^{-\frac{1}{2}}\rho')''- \tfrac{1}{2} k^2R^{-\frac{3}{2}}(R^{\frac12})''\rho'. \label{eq:integrated-u}
\end{align}
Secondly, we introduce the transformation
\begin{equation}
\label{defGama}
\begin{aligned}
\Gamma &: H^3(\R) \times H^2(\R) \to H^3(\R) \times H^2(\R),\\
\Gamma \begin{pmatrix} \rho \\ u \end{pmatrix} &:= \begin{pmatrix} \rho \\ U\rho + Ru \end{pmatrix}.
\end{aligned}
\end{equation}
Since $R > 0$ and $R$ and $U$ are uniformly bounded for all $x \in \R$, it is clear that the map $\Gamma$ is invertible. Let us denote our new perturbation variables as
\[
\begin{pmatrix} \rho \\ v \end{pmatrix} := \Gamma \begin{pmatrix} \rho \\ u \end{pmatrix} \in H^3(\R) \times H^2(\R),
\]
so that
\begin{equation}\label{eq:v}
	v(x) = U(x)\rho(x) + R(x)u(x), \qquad \qquad \mbox{ a.e. in } \R.
\end{equation}
Now, let us assume that $\lambda \neq 0$. Upon substitution and differentiation we obtain
\begin{equation}
\label{eq:du}
	Ru' = -U\rho' + v' - R\left(\frac{U}{R}\right)'\rho - \frac{R'}{R}v,
\end{equation}
and, as a consequence, equation \eqref{eq:integrated-rho} can be rewritten as
\begin{equation}
\label{eq:rho}
	\lambda \rho=s\rho'-v'+R\left(\frac{U}{R}\right)'\rho+\frac{R'}{R}v.
\end{equation}
On the other hand, multiplying \eqref{eq:v} by $\lambda \neq 0$, we infer $\lambda v=\lambda Ru+\lambda U\rho$, 
and using \eqref{eq:integrated-u}, \eqref{eq:rho} and \eqref{eq:du}, we arrive at the following spectral problem
\begin{align}
	\lambda \rho &= s\rho'-v'+g\rho+\frac{R'}{R}v, \label{eq:rho-final}\\
	\lambda v &= -f_1\rho'+f_2v'+\mu v''-f_2g\rho- \frac{R'}{R}f_2v -2\mu U'\rho' + \notag \\
	&\qquad -\mu\frac{R'}{R}v'-\mu g'\rho -\mu\left(\frac{R'}{R}\right)'v+k^2 \cL_Q\rho, \label{eq:v-final}
\end{align}
where we have introduced the functions
\begin{equation}
\label{deffg}
\begin{aligned}
	f_1(x)&:=R(x)h'(R(x))-U(x)^2, \\
	f_2(x)&:=s-2U(x),\\
	g(x)&:=R(x)\left(\frac{U(x)}{R(x)}\right)' ,
\end{aligned}
\end{equation}
and the operator $\cL_Q : \cD(\cL_Q) = H^3(\R) \subset L^2(\R) \to L^2(\R)$ is defined by
\begin{equation*}
\cL_Q\rho:=\frac12R^{\frac12}(R^{-\frac12}\rho')''-\frac12R^{-\frac12}(R^{\frac12})''\rho'.
\end{equation*}
For later use, notice that this operator can be recast as
\begin{equation}\label{eq:L_Q}
	\cL_Q \rho := \tfrac{1}{2} \rho'''- \tfrac{1}{2} \Big( \frac{R'}{R}\Big) \rho'' - \tfrac{1}{2} \Big( \frac{R'}{R} \Big)' \rho'. 
\end{equation}
Hence, the new (yet equivalent) spectral problem \eqref{eq:rho-final} and \eqref{eq:v-final} has the form
\[
\begin{aligned}
\cA &: \cD(\cL) \subset L^2(\R) \times L^2(\R) \to L^2(\R) \times L^2(\R),\\
\cA \begin{pmatrix} \rho \\ v \end{pmatrix} &:= \begin{pmatrix} s\rho'-v'+g\rho+\frac{R'}{R}v \\ \, \\ -f_1\rho'+f_2v'+\mu v''-f_2g\rho- \frac{R'}{R}f_2v -2\mu U'\rho' + \\ -\mu\frac{R'}{R}v'-\mu g'\rho -\mu\left(\frac{R'}{R}\right)'v+k^2 \cL_Q\rho\end{pmatrix},
\end{aligned}
\]
where the auxiliary (integrated) operator $\cA$ is endowed with the following property: the point spectrum of the original operator $\cL$ is contained in the point spectrum of $\cA$, except for the eigenvalue zero.
\begin{lemma}
$\ptsp(\cL) \backslash \{ 0 \} \subset \ptsp(\cA)$.
\label{lemequivsp}
\end{lemma}
\begin{proof}
Suppose that $\lambda \in \ptsp(\cL)$, $\lambda \neq 0$ with eigenfunction $(\hat{\rho}, \hat{u}) \in H^3 \times H^2$. Define $\rho$ and $u$ as the antiderivatives of $\hat{\rho}$ and $\hat{u}$ in \eqref{intvar}, respectively. Clearly, $\rho' \in H^3(\R)$ and $u' \in H^2(\R)$. In view that $\lambda \neq 0$, integration of \eqref{eq_variable_coeff} yields the exponential decay of $(\rho,u)$ as $|x|\rightarrow \infty$. This shows that $\rho, u \in L^2(\R)$. We conclude that $(\rho, u) \in \cD(\cL)$ and solve the integrated system \eqref{eq:integrated-rho} and \eqref{eq:integrated-u}. Since $\lambda \neq 0$, this implies that $(\rho,v) := \Gamma (\rho,u)$ belongs to $\cD(\cL)$ as well, and solves the spectral problem $\cA (\rho,v) = \lambda (\rho,v)$. The lemma is proved.
\end{proof}

\subsection{Estimates on the coefficients}

In this section we gather some further properties of the linearized operator around the profile. More precisely, we establish some estimates on the coefficients which are useful to close energy estimates in the forthcoming section.

\begin{lemma}
\label{lem:f_1,g}
Suppose $\gamma > 1$, $R^- > 0$  and $s \in (0, 2 c_s(R^-))$. 
If the shock amplitude, $\e=R^- - R^+ > 0$, is sufficiently small, then there exists a uniform constant $C > 0$ (independent of $\e$) such that 
\begin{equation}
\label{allproperties}
	\begin{aligned}
		0< C^{-1} \leq f_1(x)& \leq C, \\
		|f'_1(x)|& \leq C |R'(x)|, \\
		|f''_1(x)|&\leq C \e|R'(x)|,\\
		g(x)&\leq C^{-1} R'(x)<0, \\
		|g'(x)|&\leq C \e|R'(x)|, 
	\end{aligned}
\end{equation}
for any $x\in\R$.
\end{lemma}
\begin{proof}
See Appendix \ref{apendice}.
\end{proof}
Next, we define the functions
\begin{align}
a(x) &:= \frac{1}{2}\left(\frac{f_2(x)}{f_1(x)}\right)' + \frac{R'(x) f_2(x)}{R(x) f_1(x)}, \label{eq:f_a2}\\
b(x) &:= \frac{R'(x)}{R(x)} - \frac{f_2(x) g(x)}{f_1(x)}. \label{eq:f_b}
\end{align}
and we prove the following result.
\begin{lemma}
\label{lemma:choice}
Suppose $\gamma > 1$, $R^- > 0$  and $s \in (0, 2 c_s(R^-))$. Then, we can choose $\eta > 0$ so that the following holds:
if the shock amplitude, $\e=R^--R^+$, is sufficiently small, then there exist uniform constants $C_1,C_2> 0$ (independent of $\varepsilon$) such that
\begin{align}
	-g(x) - \frac{|b(x)|}{2 \eta} &\geq C_1 |R'(x)|, \label{est_hyperbolic1}\\
	a(x) - \frac{\eta}{2}|b(x)| &\geq C_2 |R'(x)|, \label{est_hyperbolic2}
\end{align}
for any $x \in \mathbb{R}$.
\end{lemma}
\begin{proof}
See Appendix \ref{apendice}.
\end{proof}
Finally, we conclude this section by denoting
\begin{equation}
\label{eq:f_a3}
	\omega(x):=\frac{s}{2}\left( \frac{1}{f_1(x)} \right)' - \frac{g(x)}{f_1(x)},
\end{equation}
and 
\begin{equation}
\label{eq:F}
	\kappa(\gamma) := \begin{dcases}
	\frac{5 - \sqrt{7 - 2 \gamma}}{2}, & 1\leq \gamma \leq 3,\\
	2, & \gamma > 3.
	\end{dcases} 
\end{equation}
We have the following result.
\begin{lemma}
\label{lemdest}
Assume $\gamma > 1$, $R^- > 0$  and $s \in (0, \kappa(\gamma)c_s(R^-))$. If the shock amplitude, $\vep = R^--R^+$, is sufficiently small then there exists a uniform constant $C_3>0$ (independent of $\varepsilon$) such that
\begin{equation}\label{eq:a_3-order}
	\omega(x)\geq C_3 |R'(x)|,
\end{equation}
for any $x \in \mathbb{R}$.
\end{lemma}
\begin{proof}
See Appendix \ref{apendice}.
\end{proof}

\section{Spectral stability}
\label{sec:spectral}


In this section we establish the spectral stability of sufficiently weak dispersive shock profiles. 

\subsection{Stability of the essential spectrum}

Lattanzio and Zhelyazov \cite{LaZ21b} proved that the essential spectrum of $\cL$ is stable provided that the end states are subsonic or sonic (see Proposition 2.1 in that reference). If one of the end states is subsonic and the shock amplitude is sufficiently small, then the subsonicity of the other end state is guaranteed, provided that the shock speed satisfies a certain bound. This is the content of the following result, whose proof is very similar to that of the linear viscosity case (see Lemma 3.3 in \cite{FPZ-press}) and we omit it.
\begin{lemma}
\label{lemessspect}
Assume $R^- > 0$ and $s \in (0,2c_s(R^-))$. 
If the shock amplitude, $\e = R^--R^+ > 0$, is sufficiently small then the $L^2$-essential spectrum of the linearized operator around the shock profile is stable. 
More precisely,
\[
\ess(\cL) \subset \{\lambda \in \mathbb{C} \, : \, \Re \lambda < 0\} \cup \{0\}.
\]
\end{lemma}
\begin{remark}
Similarly to the linear viscosity case \cite{FPZ-press}, it suffices to assume $R^- > 0$, $\vep = R^- - R^+ > 0$ sufficiently small and $s \in (0,2c_s(R^-))$ to stabilize the essential spectrum. This happens because the essential spectrum is bounded to the left of the Fredholm borders $\Sigma_\pm = \{ \lambda = \lambda_\pm(\xi) \in \C \, : \, \xi \in \R \}$, where the functions $\lambda = \lambda_\pm(\xi)$ are determined by the dispersion relations of the asymptotic operators $\cL_\pm$ defined in \eqref{asymop} (see equation (2.9) in \cite{LaZ21b}) and they have the same structure as their linear viscosity counterparts. What is important to remark here is that, also like in the linear viscosity case, there is accumulation of the essential spectrum near the origin (that is, there is no \emph{spectral gap}; see Figure 1 in \cite{LaZ21b}) and, consequently, it is not possible to apply standard exponentially decaying semigroup techniques to study the nonlinear stability of the profiles.
\end{remark}

\subsection{Energy estimates and the point spectrum}

In this section we establish an energy estimate to show that the auxiliary operator $\cA$ has stable point spectrum. The particular form of the new spectral problem, \eqref{eq:rho-final} and \eqref{eq:v-final}, the monotonicity of the density profile (Proposition \ref{lem:R-prop}), as well as the bounds for the coefficients stated in Lemmata \ref{lem:f_1,g}, \ref{lemma:choice} and \ref{lemdest}, play a key role along the proof.

\begin{lemma}[energy estimate]
\label{lemeest}
Assume $\gamma > 1$, $R^- > 0$  and $s\in(0, \kappa(\gamma)c_s(R^-))$, where $\kappa(\gamma)$ is defined in \eqref{eq:F}. 
Let the shock amplitude, $\e = R^- - R^+ > 0$, be sufficiently small. Then the operator $\cA$ is point spectrally stable, that is, $\ptsp(\cA) \subset \{ \lambda \in \C \, : \, \Re \lambda < 0\}$.
\end{lemma}
\begin{proof}
By contradiction, suppose that $\lambda \in \ptsp(\cA)$ with $\lambda \neq 0$ and $\Re \lambda \geq 0$. Then there exists an eigenfunction $(\rho,v) \in \cD(\cL)$ such that equations \eqref{eq:rho-final} and \eqref{eq:v-final} hold. Multiply equation \eqref{eq:v-final} by $v^*/f_1$ and integrate over $\R$. 
The result is
\begin{align}
	\lambda\int_\R\frac{|v|^2}{f_1}\,dx=\int_\R\left[-\rho'v^*+\frac{f_2}{f_1}v'v^*+\frac{\mu}{f_1}v''v^*-\frac{f_2g}{f_1}\rho v^*-\frac{R'f_2}{Rf_1}|v|^2-\frac{2\mu U'}{f_1}\rho'v^*\right.\notag \\
	\left.-\mu\frac{R'}{Rf_1}v'v^*-\mu\frac{g'}{f_1}\rho v^* -\frac\mu{f_1}\left(\frac{R'}{R}\right)'|v|^2+\frac{k^2}{f_1}(\cL_Q\rho)v^*\right]\,dx.\label{eq:util}
\end{align}
Substitute \eqref{eq:rho-final} and integrate by parts in order to deduce the relation
\begin{equation}
\label{util1}
\begin{aligned}
-\Re \int_\R\rho' v^*\,dx&= \Re \int_\R \rho (v')^*\,dx= \Re\int_{\R}\rho\left(s\rho' - \lambda \rho+g\rho+\frac{R'}{R}v\right)^*dx\\
&=s \, \Re \! \int_\R (\rho')^* \rho \, dx - (\Re \lambda) \int_\R |\rho|^2 \, dx +\int_\R g|\rho|^2\,dx+\Re \! \int_{\R}\frac{R'}{R}\rho v^*\,dx\\
&= \frac{s}{2} \int_\R \frac{d}{dx} \big( |\rho|^2\big) \, dx - (\Re \lambda) \int_\R |\rho|^2 \, dx +\int_\R g|\rho|^2\,dx+\Re \! \int_{\R}\frac{R'}{R}\rho v^*\,dx\\
&=- (\Re \lambda) \int_\R |\rho|^2 \, dx+\int_\R g|\rho|^2\,dx+\Re \! \int_{\R}\frac{R'}{R}\rho v^*\,dx.
\end{aligned}
\end{equation}
Similarly,
\begin{equation}\label{eq:a_2}
	\begin{aligned}
	\Re \! \int_{\R} \left[ \frac{f_2}{f_1}v'v^* -\frac{R'f_2}{Rf_1}|v|^2\right] \, dx &= 
	\int_{\R} \left[\frac{f_2}{f_1}\Re(v'v^*) -\frac{R'f_2}{Rf_1}|v|^2\right]\, dx\\
	&= \frac{1}{2} \int_\R \left[ \frac{f_2}{f_1}\frac{d}{dx} \big( |v|^2 \big)-\frac{R'f_2}{Rf_1}|v|^2\right]  \, dx\\
	&= -\int_\R \left[\frac{1}{2} \left(\frac{f_2}{f_1}\right)' +\frac{R'f_2}{Rf_1}\right]|v|^2 \, dx\\
	&=-\int_\R a|v|^2 \, dx,
	\end{aligned}
\end{equation}
where $a$ was introduced in \eqref{eq:f_a2}. Integrate by parts once again to obtain
\begin{equation}
\mu \Re \int_\R \frac{v'' v^*}{f_1}\,dx = -\mu \int_\R \frac{|v'|^2}{f_1}\,dx - \mu \Re \int_\R \left(\frac{1}{f_1}\right)' v^* v' dx.
\label{util2}
\end{equation}
Moreover,
\begin{equation}
\label{eq:b-util}
	\Re\int_\R \left(\frac{R'}{R} - \frac{f_2 g}{f_1}\right) \rho v^*\,dx=\Re\int_\R b\rho v^*\,dx
\end{equation}
where $b$ is defined in \eqref{eq:f_b}. Taking the real part of \eqref{eq:util}, substituting the identities \eqref{util1}, \eqref{eq:a_2}, \eqref{util2}, \eqref{eq:b-util} and rearranging the terms, we arrive at
\begin{equation}
\label{util3}
(\Re \lambda) \int_\R \left[ |\rho|^2 + \frac{|v|^2}{f_1}\right]\, dx + I_1 +\mu \int_\R \frac{|v'|^2}{f_1}\, dx = I_2,
\end{equation}
where
\begin{equation*}
	I_1:= -\int_\R g |\rho|^2\,dx + \int_\R a |v|^2\,dx - \Re \int_\R b \rho v^*\,dx,
\end{equation*}
and
\begin{equation}
\label{eq:beta}
	\begin{aligned}
	I_2 &:=\mu\Re \int_\R\left( -\frac{2U'}{f_1}\rho'v^* -\frac{R'}{Rf_1}v'v^*-\frac{g'}{f_1}\rho v^* -\frac1{f_1}\left(\frac{R'}{R}\right)'|v|^2 \right)\,dx \\
	&\qquad - \mu \Re \int_\R \left(\frac{1}{f_1}\right)' v^* v'\,dx +\Re \int_\R \frac{k^2}{f_1}(\cL_Q \rho)v^*\,dx.
	\end{aligned}
\end{equation}
First, in order to estimate $I_1$, use Young's inequality and obtain
\[
\begin{aligned}
	I_1 &\geq -\int_\R g |\rho|^2\,dx + \int_\R a |v|^2\,dx - \left| \int_\R b \rho v^*\,dx \right| \\
	&\geq -\int_\R g|\rho|^2\,dx + \int_\R a |v|^2\,dx -\frac{1}{2 \eta} \int_\R |b| |\rho|^2\,dx
	-\frac{\eta}{2} \int_\R |b| |v|^2\,dx \\
	&= \int_\R \left( - g - \frac{|b|}{2 \eta} \right) |\rho|^2\,dx + \int_\R\left(a - \frac{\eta |b|}{2}\right)|v|^2\,dx,
\end{aligned}
\]
for any arbitrary $\eta > 0$. By Lemma \ref{lemma:choice}, if $\vep$ is small enough we can choose $\eta> 0$ (depending on $R^-$, $\gamma$ and $s$, but independent of $\varepsilon$) such that
\begin{align}
	I_1&\geq \int_\R \left( - g - \frac{|b|}{2 \eta} \right) |\rho|^2\,dx + \int_\R\left(a - \frac{\eta |b|}{2}\right)|v|^2\,dx \notag\\
	&\geq C_1 \int_\R |R'||\rho|^2\,dx + C_2 \int_\R |R'||v|^2\,dx,
\label{hyperbolic_ineq}
\end{align}
where $C_1, C_2 >0$ do not depend on $\varepsilon$.
Next, let us rewrite the first term in $I_2$ as
\begin{equation}
-\mu\Re \int_\R \frac{2U'}{f_1}\rho' v^*\,dx = 2 \mu \Re \int_\R \left( \frac{U'}{f_1} \right)'\rho v^* \,dx + 2 \mu \Re \int_\R \frac{U'}{f_1}\rho (v')^*\,dx, \label{util5}
\end{equation}
where, once again, we have integrated by parts. Moreover,
\begin{align}
&-\mu \Re \int_\R\left(\left[\left( \frac{1}{f_1} \right)'+ \frac{R'}{R f_1}\right]v^* v'+\frac{1}{f_1}\left( \frac{R'}{R} \right)' |v|^2\right)\,dx \nonumber \\
&\qquad= - \mu \int_\R\left(\left[ \left( \frac{1}{f_1} \right)' + \frac{R'}{R f_1} \right] \Re( v^* v') 
+\frac{1}{f_1}\left( \frac{R'}{R} \right)' |v|^2\right)\,dx \nonumber \\
&\qquad= -\mu\int_\R \left(\frac12\left[\left( \frac{1}{f_1} \right)' + \frac{R'}{R f_1}\right]\frac{d}{dx}(|v|)^2
+\frac{1}{f_1}\left( \frac{R'}{R} \right)' |v|^2\right)\,dx \nonumber \\
&\qquad=\mu\int_\R \left(\frac12\left[ \left( \frac{1}{f_1} \right)' + \frac{R'}{R f_1}\right]' 
-\frac{1}{f_1}\left( \frac{R'}{R} \right)'\right) |v|^2\,dx \nonumber \\
&\qquad= \mu \int_\R \left[\frac{1}{2} \left( \frac{1}{f_1} \right)'' + \frac{1}{2}\left( \frac{R'}{R f_1} \right)'- \frac{1}{f_1}\left( \frac{R'}{R} \right)' \right] |v|^2\,dx. \label{util6}
\end{align}
Because of \eqref{eq:L_Q}, the last term in \eqref{eq:beta} can be written as
\begin{equation*}
\Re \int_\R \frac{k^2}{f_1}(\cL_Q \rho)v^*\,dx
= \frac{k^2}{2} \Re \int_\R \left[\rho''' - \frac{R'}{R}\rho'' - \left( \frac{R'}{R} \right) '\rho' \right] \frac{v^*}{f_1}\,dx.
\end{equation*}
Integration by parts yields
\begin{align}
\Re \int_\R \frac{\rho''' v^*}{f_1}\,dx&= -\Re \int_\R \frac{(v')^* \rho''}{f_1}\,dx 
- \Re \int_\R \left( \frac{1}{f_1} \right)' v^* \rho'' \,dx \notag\\
&=\Re \int_\R \frac{(v'')^* \rho'}{f_1}\,dx+\Re \int_\R \left( \frac{1}{f_1}\right)' (v')^* \rho'\,dx \notag \\ 
&\qquad - \Re \int_\R \left( \frac{1}{f_1} \right)' v^* \rho'' \,dx. \label{disp_term1}
\end{align}
Differentiating the identity \eqref{eq:rho-final}, we obtain
\begin{equation*}
v'' = s \rho'' -\lambda \rho' + g \rho' + \frac{R'}{R}v' + g'\rho + \left( \frac{R'}{R} \right)' v,
\end{equation*}
and, consequently,
\begin{align}
\Re \int_\R \frac{(v'')^* \rho'}{f_1}\,dx&= \Re \int_\R \left[ s \rho'' -\lambda\rho' + g\rho' + \frac{R'}{R}v' + g'\rho + \left( \frac{R'}{R} \right)' v \right]^* \frac{\rho'}{f_1} \,dx \nonumber\\
&= \frac{s}{2} \int_\R \frac{(|\rho'|^2)'}{f_1}\,dx - (\Re \lambda) \int_\R \frac{|\rho'|^2}{f_1}\,dx
+ \int_\R \frac{g}{f_1} |\rho'|^2\,dx \nonumber\\
&\qquad +\Re \int_\R\left[\frac{R'}{R f_1} (v')^* \rho'+ \frac{g'}{f_1}\rho^* \rho'+ \left( \frac{R'}{R} \right)' \frac{v^* \rho'}{f_1}\right]\,dx \nonumber\\
&=-\int_\R \left[\frac{s}{2} \left( \frac{1}{f_1} \right)' - \frac{g}{f_1} \right] |\rho'|^2 \,dx - (\Re \lambda) \int_\R \frac{|\rho'|^2}{f_1}\,dx \nonumber\\
&\qquad +\Re \int_\R\left[\frac{R'}{R f_1} (v')^* \rho'+ \frac{g'}{f_1}\rho^* \rho'+ \left( \frac{R'}{R} \right)' \frac{v^* \rho'}{f_1}\right]\,dx. \label{disp_term2}
\end{align}
Furthermore,
\begin{equation}
\Re \int_\R \left( \frac{1}{f_1} \right)' v^* \rho''\,dx = -\Re \int_\R \left( \frac{1}{f_1} \right)' (v')^* \rho'\,dx
- \Re \int_\R \left( \frac{1}{f_1} \right)'' v^* \rho' \,dx, \label{disp_term3}
\end{equation}
and
\begin{equation}
-\Re \int_\R \frac{R'}{R f_1} v^* \rho'' \,dx = \Re \int_\R \left(\frac{R'}{R f_1}\right)' v^* \rho' \,dx
+\Re \int_\R \frac{R'}{R f_1}(v')^* \rho'\,dx. \label{disp_term4}
\end{equation}
Then, using \eqref{disp_term1}, \eqref{disp_term2}, \eqref{disp_term3} and \eqref{disp_term4}, we obtain
\begin{equation}\label{util7}
\begin{aligned}
2\Re \int_\R (\cL_Q \rho) \frac{v^*}{f_1}\,dx&= -\int_\R \omega |\rho'|^2\,dx - (\Re \lambda) \int_\R \frac{|\rho'|^2}{f_1}\,dx \\
&\qquad + \Re \int_\R \left[\left( \frac{1}{f_1} \right) '' + \left( \frac{R'}{R f_1}\right)' \right] v^* \rho'\,dx \\
&\qquad +2\Re \int_\R \left[\frac{R'}{R f_1} + \left( \frac{1}{f_1}\right)' \right] (v')^* \rho'\,dx+\Re\int_\R\frac{g'}{f_1}\rho^* \rho'\,dx,
\end{aligned}
\end{equation}
where the function $\omega = \omega(x)$ is defined in \eqref{eq:f_a3}. Finally, using \eqref{util5}, \eqref{util6}, multiplying \eqref{util7} by $k^2/2$ and substituting the result in \eqref{util3} and in \eqref{eq:beta}, we obtain

\begin{equation*}
	(\Re \lambda) \int_\R \left[ |\rho|^2 + \frac{|v|^2}{f_1}+\frac{k^2|\rho'|^2}{2f_1}\right]+I_1+\mu\int_\R \frac{|v'|^2}{f_1}+\frac{k^2}{2}\int_\R \omega |\rho'|^2=I_2,
\end{equation*}
where $I_2$ can be recast as
\begin{equation}\label{eq:beta2}
\begin{aligned}
	I_2 &=\mu \int_\R \left[\frac{1}{2} \left( \frac{1}{f_1} \right)'' + \frac{1}{2}\left( \frac{R'}{R f_1} \right)'- \frac{1}{f_1}\left( \frac{R'}{R} \right)' \right] |v|^2\,dx\\
	&\qquad +2 \mu \Re \int_\R \left( \frac{U'}{f_1} \right)'\rho v^* \,dx + 2 \mu \Re \int_\R \frac{U'}{f_1}\rho (v')^*\,dx \\
	&\qquad-\mu \Re \int_\R \frac{g'}{f_1} v^* \rho\,dx +
	\frac{k^2}{2}\Re \int_\R \left[\left( \frac{1}{f_1} \right) '' + \left( \frac{R'}{R f_1}\right)' \right] v^* \rho'\,dx \\
	&\qquad +k^2\Re \int_\R \left[\frac{R'}{R f_1} + \left( \frac{1}{f_1}\right)' \right] (v')^* \rho'\,dx+\frac{k^2}{2}\Re\int_\R\frac{g'}{f_1}\rho^* \rho'\,dx.
\end{aligned}
\end{equation}
Now, using \eqref{hyperbolic_ineq}, we estimate $I_2$ as
\begin{equation}
\label{util8}
\begin{aligned}
	I_2 \geq&(\Re \lambda)\int_\R \left[ |\rho|^2 + \frac{|v|^2}{f_1}+\frac{k^2|\rho'|^2}{2f_1}\right] dx+ C_1 \int_\R |R'| |\rho|^2\,dx + \\
	&\qquad +C_2 \int_\R |R'| |v|^2\,dx+\mu \int_\R \frac{|v'|^2}{f_1}\, dx+\frac{k^2}{2}\int_\R \omega |\rho'|^2\,dx, 
\end{aligned}
\end{equation}
for some uniform $C_1,C_2>0$. To conclude, let us estimate all the terms appearing in \eqref{eq:beta2}: 
first, by using \eqref{eq:prop-R} and \eqref{allproperties}, we deduce
\begin{align*}
	\mu\int_\R \left|\frac{1}{2} \left( \frac{1}{f_1} \right)'' + \frac{1}{2}\left( \frac{R'}{R f_1} \right)'- \frac{1}{f_1}\left( \frac{R'}{R} \right)' \right| |v|^2\,dx
	\leq C \e \int |R'| |v|^2 \,dx,
\end{align*}
for some positive constant $C$, which does not depend on $\e$.
Similarly, thanks to \eqref{eq:prop-R}, \eqref{eq:est-U}, \eqref{allproperties} and Young's inequality, we have
\begin{align*}
	2 \mu\left| \Re \int_\R \left( \frac{U'}{f_1} \right)'\rho v^* \,dx\right|&\leq C\e \int_\R |R'| |v|^2\,dx  + C \e \int_\R |R'| |\rho|^2\,dx, \\
	2 \mu\left|\Re \int_\R \frac{U'}{f_1}\rho (v')^*\,dx\right|
	&\leq \frac{C}{\e} \int_\R |R'| |v'|^2\,dx  + C\e \int_\R |R'| |\rho'|^2\,dx \\
	&\leq C \e\int_\R |v'|^2\,dx+C \e\int_\R |R'| |\rho'|^2\,dx.
\end{align*}
Here and in what follows $C>0$ is a constant, which does not depend on $\e$ and whose value may change from line to line.
Apply estimates \eqref{allproperties} to obtain
\begin{equation*}
	\mu\left|\Re \int_\R \frac{g'}{f_1} v^* \rho\,dx\right|\leq C \e \int_\R |R'| |v|^2\,dx  + C \e\int_\R |R'| |\rho|^2\,dx.
\end{equation*}
Finally, apply Young's inequality, and by Proposition \ref{lem:R-prop} and Lemma \ref{lem:f_1,g} as before, we deduce
\begin{align*}
	\frac{k^2}{2}\left|\Re \int_\R \left[\left( \frac{1}{f_1} \right)''+\left( \frac{R'}{R f_1}\right)' \right] v^* \rho'\,dx\right|
	&\leq C\e\int_\R |R'| |v|^2\,dx +C \e\int_\R |R'| |\rho'|^2\,dx,\\
	k^2\left|\Re \int_\R \left[\frac{R'}{R f_1} + \left( \frac{1}{f_1}\right)' \right] (v')^* \rho'\,dx\right|
	&\leq \frac{C}{\e}\int_\R |R'| |v'|^2\,dx  + C\e\int_\R |R'| |\rho'|^2\,dx,\\
	&\leq C\e \int_\R |v'|^2\,dx  + C \e\int_\R |R'| |\rho'|^2\,dx,\\
	\frac{k^2}{2}\left|\Re\int_\R\frac{g'}{f_1}\rho^* \rho'\,dx\right|&\leq C \e\int_\R |R'| |\rho|^2\,dx+C\e \int_\R |R'| |\rho'|^2\,dx.
\end{align*}
Substituting all these inequalities in \eqref{util8} and using \eqref{allproperties}, \eqref{eq:a_3-order}, we end up with
\begin{equation}
\label{EE}
\begin{aligned}
	(\Re \lambda)\int_\R \left[ |\rho|^2 + \frac{|v|^2}{f_1}+\frac{k^2|\rho'|^2}{2f_1}\right] dx &+ \left( C_1 -C \e\right)\int_\R |R'||\rho|^2\,dx\\
	&+\left(C_2-C \e\right) \int_\R |R'| |v|^2\,dx\\
	&+\left( \frac{C_3 k^2}{2} -C\e\right) \int_\R |R'| |\rho'|^2\,dx\\
	&+ \left(C_1 \mu - C\e\right) \int_\R |v'|^2\,dx \, \leq 0,
\end{aligned}
\end{equation}
for uniform constants $C_j > 0$ independent of $\vep$. This is a contradiction with $\Re \lambda \geq 0$ if $\e>0$ is small enough. We conclude that $\Re \lambda < 0$ and the lemma is proved.
\end{proof}

\begin{remark}
\label{remkeyenergy}
From the energy estimate \eqref{EE} one identifies that, under the current assumptions, the appropriate energy for the perturbations is precisely 
\[
E = \| \rho \|_{L^2}^2 + \| U(x) \rho + R(x) u \|_{L_{w}^2}^2 + \tfrac{1}{2}k^2 \| \rho_x \|_{L_{w}^2}^2,
\]
in terms of weighted Sobolev norms with weight function $w(x) = 1/\sqrt{f_1(x)}$ (notice that $f_1$ is strictly positive because of \eqref{allproperties}). For the evolution problem, with perturbations depending on time, the spectral estimate \eqref{EE} suggests that $dE/dt \leq 0$.
\end{remark}

\begin{corollary}
\label{corptsp}
Under the assumptions $\gamma > 1$, $R^- > 0$  and $s\in(0, \kappa(\gamma)c_s(R^-))$, if the shock amplitude is sufficiently small then the point spectrum of the linearized operator around the wave is stable, $\ptsp(\cL) \subset \{ \Re \lambda > 0 \} \cup \{ 0 \}$.
\end{corollary}
\begin{proof}
It follows immediately from Lemma \ref{lemequivsp} and the energy estimate from Lemma \ref{lemeest}.
\end{proof}

We are now able to prove our main result.

\begin{theorem}[spectral stability of weak dispersive shock profiles]
\label{mainthm}
Suppose that $\gamma > 1$, $R^- > 0$ and $s \in (0, \bar{s})$, where
\begin{equation}
\label{sbound}
\bar{s} := \min \left\{ 2c_s(R^-), \kappa(\gamma) c_s(R^-)\right\}.
\end{equation}
Then there exists a sufficiently small $\bar{\vep} > 0$ such that, if the shock amplitude $\vep = R^- - R^+$ satisfies $0 < \vep < \bar{\vep}$ then the dispersive shock profile is spectrally stable.
\end{theorem}
\begin{proof}
Since the shock speed satisfies the bound \eqref{sbound} then it suffices to choose $\bar{\vep} > 0$ sufficiently small such that the conditions of Lemma \ref{lemessspect} and of Corollary \ref{corptsp} hold. We conclude that $\sigma(\cL) \subset \{ \lambda \in \C \, : \, \Re \lambda < 0\} \cup \{0 \}$, as claimed.
\end{proof}

\begin{remark}
\label{remnonlineargood}
It is to be observed that the condition \eqref{sbound} on the shock speed is the only assumption needed to conclude the spectral stability of sufficiently weak dispersive shock profiles. For instance, if $\gamma \geq 3$ then \eqref{sbound} is equivalent to $s \in (0, 2c_s(R^-))$ and we have stability of both essential and point spectra (notice that $\kappa(\gamma) = 2c_s(R^-)$ for $\gamma \geq 3$). The inequality $s < 2c_s(R^-)$ is equivalent to subsonicity of both end states, $|u^\pm| < c_s(R^\pm)$, which is physically meaningful. However, if $1 < \gamma < 3$ then a stronger condition is needed to conclude the stability of the point spectrum, namely, $s \in (0, \kappa(\gamma) c_s(R^-))$. This contrasts with the equivalent condition for the shock speed in the case with linear viscosity, for which we need $s \in (0, \min \{ 2c_s(R^-), \tfrac{1}{2}(\gamma+1)c_s(R^-))$ in order to conclude stability (see \cite{FPZ-press}). It is remarkable that the nonlinear viscosity allows a larger set of values for the shock speed to admit stability of weak profiles; see Figure \ref{figcomparison}. Still, and like in the linear viscosity case, we conjecture that condition \eqref{sbound} is merely a technical assumption needed in the method of proof and that sufficiently weak shock profiles are spectrally stable even for $\kappa(\gamma) c_s(R^-) \leq s < 2c_s(R^-)$ if $1 < \gamma < 3$, as numerical calculations of the point spectrum show (see \cite{LaZ21b}).

\begin{figure}[t]
\begin{center}
\includegraphics[scale=.6, clip=true]{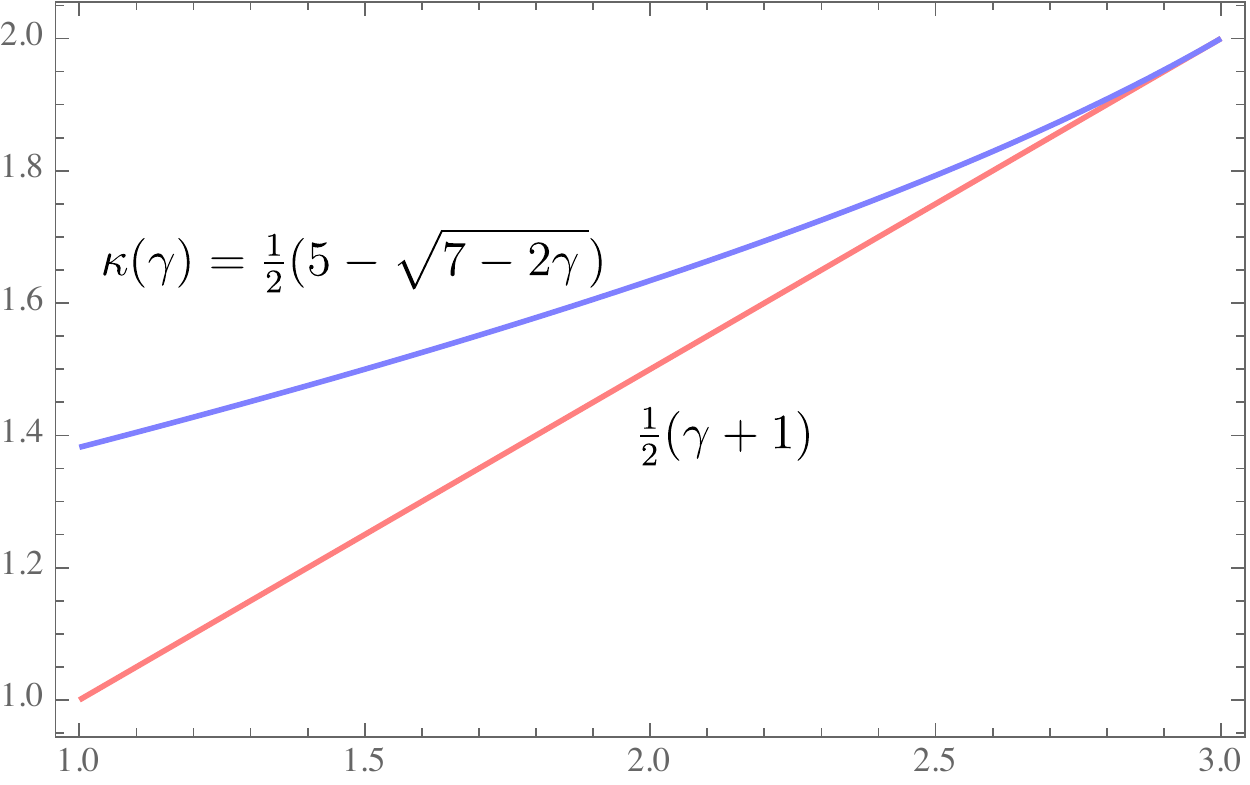}
\end{center}
\caption{\small{Comparison between the upper bounds for the shock speed that allow spectral stability, namely, $s \in (0, \kappa(\gamma) c_s(R^-))$ for the system with nonlinear viscosity under consideration, and $s \in (0, \tfrac{1}{2}(\gamma+1)c_s(R^-))$ for the associated QHD system with \emph{linear} viscosity (see \cite{FPZ-press}). The graphs (in red for the linear case, blue for the nonlinear one) show the constants, depending on the parameter $\gamma \in (1,3)$, that determine the upper bound. Notice that the nonlinear viscosity coefficient allows a larger set of shock speed to admit stability (color online).}}
\label{figcomparison}
\end{figure}

\end{remark}

\section{Discussion and open problems}

In this work, we have proved that subsonic dispersive shock profiles for system \eqref{QHD-E} with Bohmian dispersion and nonlinear viscosity are spectrally stable provided that their amplitudes are sufficiently small. In other words, the linearized operator around a sufficiently weak viscous dispersive shock profile has stable $L^2$-spectrum located in the complex half plane with negative real part (except for the eigenvalue zero). Following previous works on spectral stability of purely viscous \cite{HuZ02}, relaxation \cite{Hu03,MZ09}, or visco-capillar shocks \cite{Hu09,FPZ-press}, we establish this result by performing energy estimates at the spectral level, that is, on the spectral equations for the perturbations. In particular, we implemented a novel weighted energy function that takes into account both the nonlinear viscosity term and the quantum Bohm potential that appears in the dispersive term. Our results extend our previous analysis in the case of QHD systems with linear viscosity \cite{FPZ-press} and it is compatible with the numerical evidence of the stability of dispersive profiles with nonlinear viscosity provided by Lattanzio and Zhelyazov \cite{LaZ21b}.

This analysis can be regarded as the first step of a more general program, that aims at proving the nonlinear stability of sufficiently weak dispersive shock profiles. This is a problem that, up to our knowledge, remains open in the context of QHD systems (as far as we know, the only works that study the nonlinear stability of small-amplitude, monotone, viscous dispersive shock profiles are those of Zhang \emph{et al.} \cite{ZLY16} and the thesis by Khodja \cite{KhdPhD89}, for compressible fluids of Korteweg type with constant viscosity and capillarity coefficients). Due to the absence of a spectral gap (a feature shared by purely viscous shocks) this problem requires a detailed analysis of the semigroup generated by the linearized operator, either for the application of pointwise Green's function bounds just like in the purely viscous case \cite{ZH98,MZ03} (see, for example, the analysis of Howard and Zumbrun \cite{HZ00} for viscous-dispersive \emph{scalar} equations; the systems case remains open, as far as we know), or for adopting a more classical PDE approach (cf. \cite{PaW04,ZLY16}).

Finally, it is worth mentioning that monotonicity is a crucial property that links our analysis in the dispersive case with the classical methods implemented for purely viscous shock profiles. When the shock amplitude increases and the dispersive term plays a more significant role, the profiles are no longer monotone and begin to exhibit an oscillatory behavior. The numerical (spectral) analysis of \cite{LaZ21b}, however, applies also to non-monotone shocks and constitute evidence of their stability. Proving analytically the stability of viscous dispersive shocks beyond the small-amplitude regime is a particularly difficult open problem. Partial results that consider \emph{weak} oscillations are available \cite{KhdPhD89}, but the stability of large-amplitude (and hence highly oscillating) shocks, or the onset for instability based on the shock amplitude, are far from being understood. Thanks to the ``stabilizing'' properties of the nonlinear viscosity term (see Remark \ref{remnonlineargood}), we believe that the stability of oscillatory shock profiles for the QHD system under consideration in this paper (or an equivalent system with nonlinear viscosity for compressible fluids, for that matter) is a topic of research that deserves further examination.

\section*{Acknowledgements}
The authors are grateful to Corrado Lattanzio for useful conversations. The work of D. Zhelyazov was supported by a Post-doctoral Fellowship by the Direcci\'{o}n General de Asuntos del Personal Acad\'{e}mico (DGAPA), UNAM. The work of R. Folino was partially supported by DGAPA-UNAM, program PAPIIT, grant IA-102423.
The work of R. G. Plaza was partially supported by DGAPA-UNAM, program PAPIIT, grant IN-104922.

\appendix
\section{Coefficient estimates}
\label{apendice}

\subsection{Proof of Lemma \ref{lem:f_1,g}}

Thanks to the definition of the enthalpy \eqref{eq:ent} and equality \eqref{eq:U-R} let us rewrite $f_1$ and $g$ as
\begin{align}
	f_1(x)&=\gamma R(x)^{\gamma-1}-\left(s-\frac{A}{R(x)}\right)^2=F_1(A,R(x)),\label{eq:f1}\\
	g(x)&=\frac{AR'(x)}{R^2(x)}-\frac{sR'(x)}{R(x)}+\frac{AR'(x)}{R^2(x)}=G(A,R(x))R'(x),\label{eq:g}
\end{align}
where
\[
	F_1(A,R):= \gamma R^{\gamma-1} - \left(s-\frac{A}{R}\right)^2,\quad G(A,R):= \frac{2A-sR}{R^2}.
\]
Using \eqref{expansion_A}, let us define $A^- := R^- c_s(R^-)$, so that $A = A^- + O(\e)$ and we have $F_1(A^-,R^-)=2sc_s(R^-)-s^2=s(2c_s(R^-)-s)>0$, if $s \in (0, 2 c_s(R^-))$.
Therefore, if we choose $\e = R^- - R^+>0$ sufficiently small then $f_1(x) \geq\theta_1> 0$, for any $x\in\R$ and some $\theta_1> 0$. 
Moreover, since $f_1$ has finite limits as $x \to \pm \infty$ and $R \in [R^{+}, R^{-}]$ 
we also obtain $\theta_1^{-1} \geq f_1(x) \geq\theta_1$ for some uniform $\theta_1\in(0,1)$ and any $x \in \R$. 

On the other hand, since $f'_1(x)=\partial_RF_1(A,R(x))R'(x)$, with $R'(x)<0$ for any $x\in\R$, we obtain
$|f'_1(x)|\leq|\partial_RF_1(A,R(x))||R'(x)|\leq\theta_2|R'(x)|$ for any $x\in\R$, provided $\e>0$ is small, where the constant $\theta_2>0$ is the maximum of $|\partial_R F_1(A,R)|$ in a small neighborhood of $(A^-,R^-)$
(we again use that $A=A^-+O(\e)$ and $R(x)\in[R^+,R^-]$, for any $x\in\R$).

Similarly, it suffices to compute $f''_1$ and apply inequalities \eqref{eq:prop-R} to obtain $|f_1''(x)| \leq \theta_3 \vep |R'(x)|$ for some uniform $\theta_3 > 0$, provided that $\e$ is sufficiently small. 
	
Regarding the function $g$, since
$$G(A^-,R^-)=\frac{2 c_s(R^-) -s}{R^-}>0,$$
if $s \in (0, 2 c_s(R^-))$, we can choose $\e>0$ sufficiently small so that
$G(A,R(x))\geq\theta_4$ for any $x\in\R$ and 
for some $\theta_4 > 0$ independent on $\e$. This implies $g(x)\leq \theta_4 R'(x)<0$.
Upon differentiation we infer
$g'(x)=G(A,R(x))R''(x)+\partial_RG(A,R(x))R'(x)^2$,
and from \eqref{eq:prop-R} we deduce that $|g'(x)| \leq \theta_5 \vep |R'(x)|$ for some uniform $\theta_5 > 0$.
Finally, one may choose a uniform constant $C \geq \max \{ \theta_1^{-1}, \theta_2, \theta_3, \theta_4^{-1}, \theta_5 \} > 0$, independent of $\vep$, to get the result.
\qed

\subsection{Proof of Lemma \ref{lemma:choice}}
First of all, thanks to the equality \eqref{eq:U-R} let us rewrite $f_2$ as
\begin{equation}
\label{eq:f2}
	f_2(x):=-s+\frac{2A}{R(x)}=F_2(A,R(x)),
\end{equation}
where $F_2(A,R):= -s+\frac{2 A}{R}$. Consequently, substituting \eqref{eq:f1}, \eqref{eq:g} and \eqref{eq:f2} in the definitions of $a$ and $b$, \eqref{eq:f_a2} and \eqref{eq:f_b}, we obtain
\begin{equation*}
	a(x)=\frac{f_2'(x)f_1(x)-f_2(x)f'_1(x)}{2f_1(x)^2}+\frac{R'(x) f_2(x)}{R(x) f_1(x)}=- H(A,R(x))R'(x),
\end{equation*}
with
\begin{equation*}
	H(A,R) = -\frac{\partial_R F_2(A,R) F_1(A,R) - F_2(A,R) \partial_R F_1(A,R)}{2F_1(A,R)^2} - \frac{F_2(A,R)}{R F_1(A,R)},
\end{equation*}
and 
\begin{equation*}
	b(x)=R'(x)\left[\frac{1}{R(x)}-\frac{f_2(x)G(A,R(x))R'(x)}{f_1(x)}\right]=R'(x)B(A,R(x)),
\end{equation*}
where
\begin{equation*}
	B(A,R) = \frac{1}{R} - \frac{F_2(A,R) G(A,R)}{F_1(A,R)}.
\end{equation*}
Similarly to the proof of \eqref{allproperties}, we profit from expansion \eqref{expansion_A} and define $A^- := R^- c_s(R^-)$, so that $A = A^- + O(\e)$, and let $F_2(A^-,R^-):= -s+2c_s(R^-)=F_1(A^-,R^-)/s$. Let us compute
\begin{align*}
	G(A^-,R^-)&=\frac{2 c_s(R^-) -s}{R^-}=:q_1,\\
	H(A^-,R^-)&=\frac{c_s(R^-)^2(\gamma+1) - 4 c_s(R^-) s + 2 s^2}{2 R^- s^2 (2 c_s(R^-) - s)}=:q_2,\\
	B(A^-,R^-)&= - \frac{2(c_s(R^-) - s)}{R^- s}=:q_3.
\end{align*}
Since $s \in (0, 2 c_s(R^-))$, we have $q_1 > 0$. 
Moreover, since $\gamma > 1$, we deduce
\[
c_s(R^-)^2(\gamma+1) - 4 c_s(R^-) s + 2 s^2 > 2 c_s(R^-)^2 -4 c_s(R^-) s + 2 s^2 = 2(c_s(R^-) - s)^2 \geq 0;
\]
therefore, the condition $s \in (0, 2 c_s(R^-))$ implies $q_2 > 0$ as well.
	
We are now ready to prove \eqref{est_hyperbolic1} and \eqref{est_hyperbolic2}. Indeed, since $R'(x)<0$ for any $x\in\R$, one has
\begin{align*}
	-g(x) - \frac{|b(x)|}{2 \eta}&= -G(A,R(x))R'(x) - \frac{|R'(x)B(A,R(x))|}{2 \eta}\\
	&= \left(G(A,R(x)) - \frac{|B(A,R(x))|}{2 \eta}\right)|R'(x)|,
\end{align*}
and
\begin{align*}
	a(x) - \frac{\eta}{2}|b(x)|&= -H(A,R(x))R'(x) - \frac{\eta}2|R'(x)B(A,R(x))|\\
	&= \left(H(A,R(x)) - \frac{\eta}2|B(A,R(x))|\right)|R'(x)|.
\end{align*}
We claim that one can choose $\eta>0$ such that 
\begin{align}
	G(A^-,R^-) - \frac{|B(A^-,R^-)|}{2 \eta}&=q_1 - \frac{|q_3|}{2 \eta}> 0, \label{ineq_eta_1}\\
	H(A^-,R^-) - \frac{\eta}2|B(A^-,R^-)|&=q_2 - \frac{\eta}{2}|q_3|> 0. \label{ineq_eta_2}
\end{align}
Inequalities \eqref{ineq_eta_1} and \eqref{ineq_eta_2} guarantee \eqref{est_hyperbolic1} and \eqref{est_hyperbolic2} for $\e>0$ sufficiently small. Indeed, thanks to $R^+=R^--\e$ and $A=A^-+O(\e)$ we can state that there exists $\e_0>0$ small enough such that if $\e\in(0,\e_0)$, then
\begin{align}
	G(A,R(x)) - \frac{|B(A,R(x))|}{2 \eta}&\geq C_1, \label{eq:end1}\\
	H(A,R(x)) - \frac{\eta}2|B(A,R(x))|&\geq C_2, \label{eq:end2}
\end{align}
for any $x\in\R$, with $C_1,C_2>0$ and independent on $\e$.
Thus, it remains to prove that \eqref{ineq_eta_1} and \eqref{ineq_eta_2} hold true for some $\eta>0$.
Notice that if $s=c_s(R^-)$ then $q_3=0$, \eqref{ineq_eta_1} and \eqref{ineq_eta_2} hold true for any $\eta>0$, because we already proved that $q_1>0$ and $q_2>0$.
Hence, if $s=c_s(R^-)$, for any $\eta>0$ there exists $\e_0>0$ small enough such that if $\e\in(0,\e_0)$ then \eqref{eq:end1} and \eqref{eq:end2} hold true for any $x\in\R$.

On the other hand, if $s\neq c_s(R^-)$, the inequalities \eqref{ineq_eta_1} and \eqref{ineq_eta_2} are equivalent to
\begin{equation*}
	\frac{|q_3|}{2 q_1} < \eta < \frac{2 q_2}{|q_3|},
\end{equation*}
and, therefore, we only need to prove that
\begin{equation}\label{ineq_eta_3}
	\frac{|q_3|}{2 q_1} < \frac{2 q_2}{|q_3|}.
\end{equation}
The latter inequality is equivalent to $q_3^2<4q_1q_2$, that is
\[
\frac{4(c_s(R^-) - s)^2}{(R^- s)^2}<4\left[\frac{2c_s(R^-) -s}{R^-}\right]\frac{c_s(R^-)^2(\gamma+1) - 4 c_s(R^-) s + 2 s^2}{2 R^- s^2 (2 c_s(R^-) - s)}.
\]
Simplifying this expression we obtain
\[
(c_s(R^-) - s)^2<\frac{c_s(R^-)^2(\gamma+1) - 4 c_s(R^-) s + 2 s^2}2,
\]
namely $c_s(R^-)^2 < \tfrac{1}{2} c_s(R^-)^2(\gamma+1)$.
Therefore, if $\gamma>1$ then \eqref{ineq_eta_3} holds true.
In conclusion, we showed that for any fixed
\[
\eta\in\left(\frac{|q_3|}{2 q_1},\frac{2 q_2}{|q_3|}\right),
\]
there exists $\e_0>0$ sufficiently small such that if $\e\in(0,\e_0)$, then \eqref{eq:end1} and \eqref{eq:end2} hold true for any $x\in\R$,
and the proof is complete.
\qed

\subsection{Proof of Lemma \ref{lemdest}}
Let us proceed as in the proof of Lemma \ref{lemma:choice}.
Substituting \eqref{eq:f1} and \eqref{eq:g} in \eqref{eq:f_a3}, we deduce
\[
	\omega(x)=-\frac{s\partial_R F_1(A,R(x))R'(x)}{2F_1(A,R(x))^2} \frac{G(A,R(x))R'(x)}{F_1(A,R(x))} = D(A,R(x))R'(x),
\]
where
\begin{align*}
	D(A,R):=\frac{1}{2 R F_1(A,R)^2}&\left[ - s R \left( \gamma(\gamma-1)R^{\gamma-2}- 2 \left( s - \frac{A}{R} \right) \frac{A}{R^2} \right)\right.\\
	&\qquad\qquad \left. + 2 F_1(A,R) \left(s - \frac{2 A}{R} \right)\right].
\end{align*}
After some algebra we arrive at
\begin{equation*}
	D(A^-,R^-) = - \frac{2 s^2 - 10 c_s(R^-) s + (\gamma + 9) c_s(R^-)^2}{2 R^- s (2 c_s(R^-) - s)^2},
\end{equation*}
where $A^-=R^-c_s(R^-)$ as before.
Notice that the denominator is positive for $s \in (0, 2 c_s(R^-))$, while the numerator is a quadratic function of $s$. 
Let us consider two cases: for $1 < \gamma < 3$, we have
\[
\Big{\{} s \in (0, 2 c_s(R^-)) : 2 s^2 - 10 c_s(R^-) s + (\gamma + 9)  c_s(R^-)^2 > 0 \Big{\}} = \left(0,\kappa(\gamma)c_s(R^-) \right).
\]
On the other hand, if $\gamma \geq 3$, then the numerator is positive for any $s \in (0, 2 c_s(R^-))$. 
Hence, for $s \in (0,\kappa(\gamma)c_s(R^-))$, we have $D(A^-,R^-)<0$. 
Therefore, if $\e > 0$ is sufficiently small, then there exists a constant $C_3 > 0$ (independent of $\e$), such that $D(A,R(x)) \leq -C_3$ for any $x\in\R$.
Since $R'(x) < 0$, for any $x\in\R$ we conclude that 
$$\omega(x)=D(A,R(x)) R'(x) \geq C_3 |R'(x)|,$$ 
for any $x \in \R$ and the proof is complete.
\qed

\def\cprime{$'\!\!$}\def\cprime{$'\!\!$}\def\cprime{$'\!\!$}








%
%
%
%
%
%

\end{document}